\documentclass[11pt]{article}
\pdfoutput=1

\def\pb{}

\usepackage[dvipsnames]{xcolor}

\def\beq{\begin{equation} }\def\eeq{\end{equation} }\def\1{\mathbf{1}}

\usepackage[framemethod=default]{mdframed}
\usepackage{caption}
\usepackage{indentfirst}
\usepackage{bm, mathrsfs, graphics,float,amssymb,amsmath,subeqnarray,setspace,graphicx,amsthm,epstopdf,subfigure, enumerate, color}
\usepackage[utf8]{inputenc}
\usepackage[colorlinks,
      linkcolor=red,
      anchorcolor=blue,
      citecolor=blue
      ]{hyperref}
\usepackage{natbib}

\usepackage{fullpage}
\numberwithin{equation}{section}

\newtheorem{lemma}{Lemma}
\newtheorem{theorem}{Theorem}

\newtheorem{corollary}[theorem]{Corollary}
\newtheorem{remark}{Remark}

\ifx\assumption\undefined
\newtheorem{assumption}{Assumption}
\fi

\newcommand{\cO}{\mathcal{O}}

\newcommand{\EE}{\mathbb{E}}
\newcommand{\RR}{\mathbb{R}}

\newcommand{\bZ}{\bm{Z}}

\newcommand{\xb}{\mathbf{x}}

\newcommand{\ub}{\bm{u}}

\newcommand{\cF}{\mathcal{F}}

\newcommand{\x}{\bm{x}}
\newcommand{\z}{\bm{z}}

\usepackage{bbm}

\def\cK{\mathcal{K}}

\newcommand{\db}{\bm{d}}
\newcommand{\Ub}{\bm{U}}

\newcommand{\y}{\bm{y}}

\newcommand{\Cd}{C_{d,\delta}}

\newcommand{\cN}{\mathcal{N}}
\newcommand{\cE}{\mathcal{E}}
\newcommand{\bb}{\bm{b}}
\newcommand{\heta}{\hat{\eta}}
\newcommand{\tbZ}{\widetilde{\bZ}}
\newcommand{\Us}{U_s}

\usepackage{multirow}
\usepackage{tablefootnote}
\usepackage{colortbl}
\usepackage{hhline}

\usepackage{algorithm}
\usepackage{algorithmic}

\begin{document}
\title{
Explicit and Non-asymptotic Query Complexities of
Rank-Based Zeroth-order Algorithm on Stochastic Smooth Functions
}

\author{
	Haishan Ye
	\thanks{
		Xi'an Jiaotong University;
		email: hsye\_cs@outlook.com}
}
\date{
	\today}

\maketitle

\def\RB{\RR}
\def\TH{\tilde{H}}
\newcommand{\ti}[1]{\tilde{#1}}
\def\diag{\mathrm{diag}}
\newcommand{\norm}[1]{\left\|#1\right\|}
\newcommand{\dotprod}[1]{\left\langle #1\right\rangle}
\def\EB{\EE}
\def\tr{\mathrm{tr}}

\begin{abstract}
Zeroth-order (ZO) optimization with ordinal feedback has emerged as a fundamental problem in modern machine learning systems, particularly in human-in-the-loop settings such as reinforcement learning from human feedback, preference learning, and evolutionary strategies. 
While rank-based ZO algorithms enjoy strong empirical success and robustness properties, their theoretical understanding, especially under stochastic objectives and standard smoothness assumptions, remains limited. 
In this paper, we study rank-based zeroth-order optimization for stochastic functions where only  ordinal feedback of the stochastic function is available. 
We propose a simple and computationally efficient rank-based ZO algorithm. 
Under standard assumptions including smoothness, strong convexity, and bounded second moments of stochastic gradients, we establish explicit non-asymptotic query complexity bounds for both convex and nonconvex objectives. 
Notably, our results match the best-known query complexities of value-based ZO algorithms, demonstrating that ordinal information alone is sufficient for optimal query efficiency in stochastic settings. Our analysis departs from existing drift-based and information-geometric techniques, offering new tools for the study of rank-based optimization under noise. These findings narrow the gap between theory and practice and provide a principled foundation for optimization driven by human preferences.
\end{abstract}


\pb\section{Introduction}

Zeroth-order (ZO) optimization, the task of minimizing functions when gradient information is unavailable, has become increasingly important across machine learning \citep{loshchilov2016cma}, reinforcement learning \citep{igel2003neuroevolution}, and human-in-the-loop decision systems \citep{tangzeroth}. 
Traditional ZO methods rely on numerical function values to approximate descent directions, but many modern applications provide only ordinal feedback rather than scalar outputs. Examples include evolutionary strategies \citep{hansen2001completely}, comparative preference elicitation \citep{yue2009interactively}, online ranking platforms \citep{dwork2001rank}, and, notably, human evaluation of AI outputs \citep{bai2022training}. These settings challenge classical optimization frameworks and motivate the design of algorithms that can optimize via  non-numeric feedback.

A particularly striking instantiation arises in Reinforcement Learning from Human Feedback (RLHF): humans provide relative rankings instead of explicit scores when judging outputs, reducing cognitive burden and enabling preference modeling at scale \citep{ouyang2022training,liu2023languages,bai2022training}. RLHF has powered improvements in aligning large language and generative models with human intentions, but also exposes an underlying optimization challenge: how to efficiently optimize a black-box objective when only ranking or preference oracles are available to evaluate candidate solutions.

Rank-based optimization sits at the center of this landscape. Algorithms that exploit only order information, such as ranking or comparison oracles, are attractive for their robustness to noise and invariant behavior under monotonic transformations of the underlying objective. 
This robustness is one reason why rank-based strategies underlie successful heuristics like CMA-ES \citep{hansen2001completely} and natural evolution strategies \citep{wierstra2014natural} in evolutionary computation. 
Yet, despite their empirical success, theoretical foundations for rank-centric ZO methods have been incomplete: existing convergence analyses typically focus on asymptotic behavior or narrow models of evolutionary heuristics, offering limited insight into finite-sample efficiency. 

The recent work by \citet{tangzeroth} addresses optimizing using ranking oracles, which well-aligns with human preference data and other ordinal signals. 
They propose ZO-RankSGD, a novel zeroth-order algorithm driven by a rank-based random estimator that produces descent directions from ordinal feedback alone, and they prove its convergence to stationary points. 
However, ZO-RankSGD requires  constructing a directed acyclic graph of $\cO(N^2)$ edges where $N$ is the sample size.
This makes the construction of descent direct computation inefficient. 
Though  \citet{tangzeroth} show ZO-RankSGD can find an $\varepsilon$-stationary point at cost at most $\cO\left(\frac{d}{\varepsilon^2}\right)$, their analysis does \emph{not} based on the standard assumptions of stochastic gradient descent, such as the stochastic gradient $\nabla f(\x;\;\xi)$  being of bounded second moment. 
Another weakness is that their convergence analysis can \emph{not} be used to obtain the convergence rate when the objective function is smooth and convex.   

In parallel, a substantial body of work in online bandit optimization studies ZO algorithms with comparison oracles, which can be viewed as a special case of rank-based feedback \citep{yue2009interactively, kumagai2017regret, zhu2023principled}. A key limitation of these approaches is that their analyses are intrinsically tied to pairwise comparisons and do not extend to richer ordinal information such as full rankings, which are particularly relevant in human-in-the-loop scenarios. Furthermore, these works typically assume convex or strongly convex objectives, limiting their applicability to the nonconvex settings prevalent in modern machine learning.

Complementing these lines of research, \citet{ye2025explicit} investigates rank-based ZO algorithms that operate on the order statistics of sampled directions and derive the first explicit non-asymptotic query complexity bounds for a class of rank-based updates under smoothness assumptions. Their results demonstrate that, with carefully designed sampling and selection mechanisms, rank-only methods can approach the efficiency of value-based ZO algorithms in terms of query complexity. However, their analysis assumes access to deterministic function values $f(\mathbf{x})$. In contrast, our work focuses on stochastic optimization settings where only noisy evaluations $f(\mathbf{x};\xi)$ are available, with $\mathbb{E}_\xi[f(\mathbf{x};\xi)] = f(\mathbf{x})$. Unfortunately, the techniques developed in \citet{ye2025explicit} do not readily extend to this stochastic regime.

Despite the recent progress outlined above, rank-based zeroth-order optimization remains far from being fully understood, and many fundamental questions are still open. From a theoretical standpoint, a central challenge lies in developing unified convergence analyses that accommodate both rich ordinal feedback (beyond pairwise comparisons) and stochastic objective evaluations under standard assumptions such as smoothness and bounded variance. Existing works typically address only isolated aspects of this problem—either deterministic objectives with full rankings, or stochastic objectives with limited comparison oracles—leaving a substantial gap between theory and practice.

Furthermore, from a broader perspective, it is still unclear what information is fundamentally necessary and sufficient for efficient zeroth-order optimization. Ranking feedback occupies an intermediate position between full value queries and binary comparisons, yet the precise trade-offs between feedback richness, query complexity, and convergence rates are not well characterized. Clarifying these trade-offs is essential for principled algorithm design and for understanding the limits of optimization driven by human feedback.

Addressing the above challenges is critical for advancing both the theory and practice of optimization with ordinal information. 
This paper tries to solve these open problems and summarize our contribution as follows:
\begin{itemize}
	\item We propose a simple rank-based ZO algorithm described in Algorithm~\ref{alg:SA1}. 
	Compared with ZO-RankSGD \citep{tangzeroth} which requires to construct a directed acyclic graph of $\cO(N^2)$ edges and takes $\cO(d N^2)$ computation cost for each iteration,
	our algorithm does \emph{not} need any graph and takes $\cO(d N)$ computation cost.
	\item Condition on $L$-smoothness, $\mu$-strong convex, and $G_u^2$-bounded second moment of stochastic gradient, we show that Algorithm~\ref{alg:SA1} can achieve a query complexity $\cO\left( \frac{dLG_u^2}{\mu^2\varepsilon}\right)$. 
	If the objective function may be nonconvex, to find an $\varepsilon$-stationary point, the query complexity of our algorithm is $\cO\left(\frac{dLG_u^2}{\varepsilon^2}\right)$. 
	To the best of the author's knowledge, these query complexities for rank-based ZO over stochastic functions are first proved.
	\item 
	The obtained explicit query complexities are the same to the query complexities of valued-based ZO algorithms \citep{wangadvancement,ghadimi2013stochastic}.
	This proves that for the ZO over stochastic functions, the ordinal information is sufficient to achieve the same query efficiency as the value-based ZO algorithms.
	\item The analysis techniques in this paper are novel and do not rely on  drift analysis \citep{he2001drift} and information-geometric optimization \citep{ollivier2017information}, which are popular in the convergence analysis of evolution strategies algorithms. 
	The technique in this paper is also not an easy extension of the one in \citet{ye2025explicit}.
	Thus, we believe the analysis
	technique has potential in the analysis of other rank-based algorithms on stochastic functions.
\end{itemize}

\section{Notation and Preliminaries}

First, we will introduce the notation of $\norm{\cdot}$ which is defined as $\norm{\x} \triangleq \sqrt{\sum_{i=1}^{d} x_i^2}$ for a $d$-dimension vector $\x$ with $x_i$ being $i$-th entry of $\x$. 
Letting $\Ub\in\RR^{m\times n}$ be a matrix, the notation $\norm{\Ub}$ is the spectral norm of $\Ub$.
Given two vectors $\x,\y \in \RR^d$, we use $\dotprod{\x,\y} = \sum_{i=1}^{d}x_iy_i$ to denote the inner product of $\x$ and $\y$. 
Next, we introduce the set $\cK$  defined as $\cK := \{1,\dots, N/4, 3N/4 + 1,\dots, N\}$ where $N$ is the sample size for each iteration.

Next, we will introduce the objective function, which we aim to solve in this paper
\begin{equation}\label{eq:det}
	\min_{\x \in \RR^d} f(\x) = \EE_\xi\left[f(\x;\;\xi)\right]
\end{equation}
where $f(\x)$ and $f(\x;\;\xi)$ are  $L$-smooth functions. 
We will also introduce several widely used assumptions about the objective function.

\begin{assumption}\label{ass:L}
	A function $f(\x)$ is called $L$-smooth with $L>0$ if for any $\x,\y\in\RR^d$, it holds that
	\begin{equation}\label{eq:L}
		| f(\y) - f(\x) - \dotprod{\nabla f(\x), \y -\x} | \leq\frac{L}{2}\norm{\y - \x}^2.
	\end{equation}
\end{assumption}

\begin{assumption}
	A function $f(\x)$ is called $\mu$-strongly convex with $\mu>0$ if for any $\x,\y\in\RR^d$, it holds that
	\begin{equation}
		f(\y) - f(\x) - \dotprod{\nabla f(\x), \y -\x} \geq \frac{\mu}{2} \norm{\y - \x}^2.
	\end{equation}
\end{assumption}
The $L$-smooth and $\mu$-strongly convex assumptions are standard in the convergence analysis of optimization \citep{nesterov2013introductory}.

Next, we will introduce the assumptions with respect to the stochastic gradient $\nabla f(\x;\;\xi)$.
\begin{assumption}[Bounded Second Moment]\label{ass:bsm}
There exists a constant $G>0$ such that
\begin{equation}
	\EE_\xi\left[\norm{\nabla f(\x;\;\xi)}^2\right] \leq G_u^2 \mbox{ for all }\x.
\end{equation}
\end{assumption}
\begin{assumption}\label{ass:lb}
	We also assume that $\norm{\nabla f(\x\;\xi)} \geq G_\ell$ with $G_\ell >0$ for all $\x$.
\end{assumption}
Assumption~\ref{ass:bsm} is a popular assumption in the convergence analysis of stochastic gradient descent \citep{nemirovski2009robust,moulines2011non}. 
Assumption~\ref{ass:lb} is also reasonable because the stochastic gradient  $\nabla f(\x;\;\xi)$ never vanish even $\x_t$ tends to the optimal point \citep{bottou2018optimization,moulines2011non}.
Furthermore, Assumption~\ref{ass:lb} is only used to set the smooth parameter of Algorithm~\ref{alg:SA1}, which does not affect the query complexity of our algorithm.

\begin{algorithm}[t]
	\caption{Rank-Based Zeroth-order Algorithm for Stochastic Smooth Function}
	\label{alg:SA1}
	\begin{small}
		\begin{algorithmic}[1]
			\STATE {\bf Input:}
			Initial vector $x_1$, smooth parameter $\alpha > 0$,  sample size $N$, and step size $\eta_t$, the rank oracle $\mathrm{Rank\_Oracle}(\cdot)$.
			\FOR {$t=1,2,\dots, T$ }
			\STATE Generate $N$ random Gaussian vectors $\ub_i$ with $i = 1,\dots, N$. 
			\STATE Access the rank oracle $\mathrm{Rank\_Oracle}(\x_t + \alpha \ub_1,\dots, \x_t + \alpha \ub_N)$ and obtain an index $(1),\dots,(N)$ satisfying
			\begin{equation*}
				f(\x_t + \alpha \ub_{(1)};\;\xi_t) \leq f(\x_t + \alpha \ub_{(2)};\;\xi_t)\leq \dots  \leq \dots f(\x_t + \alpha \ub_{(N)};\;\xi_t).
			\end{equation*}
			\STATE Set $w_{(k)}^+ = \frac{4}{N} $ with $k = 1,\dots, \frac{N}{4}$ and $w_{(k)}^- =-\frac{4}{N}$ with $k = \frac{3N}{4}+1\dots, N$ such that $\sum w_{(k)}^+ = 1$ and $\sum w_{(k)}^- = -1$. 
			\STATE Construct the descent direction
			\begin{equation*}
				\db_t 
				= 
				\sum_{k=1}^{N/4} w_{(k)}^+ \ub_{(k)}
				+ \sum_{k=3N/4+1}^{N} w_{(k)}^- \ub_{(k)}
				=
				\frac{4}{N}\sum_{k=1}^{N/4} w_{(k)}^+ \ub_{(k)}
				-
				\frac{4}{N}\sum_{k=3N/4+1}^{N} w_{(k)}^- \ub_{(k)}.
			\end{equation*}
			\STATE Update as 
			\begin{equation*}
				\x_{t+1} = \x_t + \eta_t  \db_t.   
			\end{equation*}
			\ENDFOR
			\STATE {\bf Output:} $\x_{T+1}$
		\end{algorithmic}
	\end{small}
\end{algorithm}

\section{Algorithm Description}

We generate $N$ random Gaussian vectors $\ub_i$ with $i = 1,\dots, N$, that is, $\ub_i \sim \cN(0, \bm{I}_d)$. 
Accordingly, we can get $N$ points $\x_t + \alpha \ub_i$'s. 
Then, we send these points to the rank oracle, that is, $\mathrm{Rank\_Oracle}(\x_t + \alpha \ub_1,\dots, \x_t + \alpha \ub_N)$.
Then we obtain an index $(1),\dots, (N)$ such that
\begin{equation}\label{eq:sort}
	f(\x_t + \alpha \ub_{(1)};\;\xi_t) \leq f(\x_t + \alpha \ub_{(2)};\;\xi_t)\leq \dots  \leq \dots f(\x_t + \alpha \ub_{(N)};\;\xi_t).
\end{equation}

For $1\le k \leq N/4$, these points $\x_t + \alpha \ub_{(k)}$ achieve the $\frac{N}{4}$ smallest values. 
We believe that these $\ub_{(k)}$ are probably descent directions.  
Thus, we  give positive weights on these directions. 
Accordingly, we set $w_{(k)}^+ > 0$ with $k = 1,\dots, N/4$ and $\sum w_{(k)}^+ = 1$. 
On the other hand, for $ N/4 + 1\le k \leq N$, 
these points $\x_t + \alpha \ub_{(k)}$ achieve the $\frac{N}{4}$ largest values.
We believe that these $\ub_{(k)}$ are probably ascent directions.  
Thus, we want to give negative weights on these directions. 
Accordingly, we set $w_{(k)}^- < 0 $ with $k = 3N/4+1\dots, N$ and $\sum w_{(k)}^- = -1$.
In this paper, we set $w_{(k)}^+ = \frac{4}{N}$ and $w_{(k)}^- = - \frac{4}{N}$.
This weight strategy is also very popular \citep{hansen2016cma}.

After obtaining these weights $w_{(k)}$'s, we can obtain the descent direction 
\begin{equation}\label{eq:dt}
	\db_t 
	= \sum_{k=1}^{N/4} w_{(k)}^+ \ub_{(k)}
	+ \sum_{k=3N/4+1}^{N} w_{(k)}^- \ub_{(k)}
	=
	\db_t^+ 
	+
	\db_t^-.
\end{equation}

Finally, we will update $\x_t$ along the direction $\db_t$ with a step size $\eta_t$ and obtain $\x_{t+1}$ as follows:
\begin{equation}\label{eq:update}
	\x_{t+1} = \x_t + \eta_t \db_t.
\end{equation}  
The detailed algorithm procedure is listed in Algorithm~\ref{alg:SA1}.

\section{Convergence Analysis and Query Complexities}

First, we will introduce several important notations. 
The first two  are 
\begin{equation}\label{eq:z}
	\z_t := \frac{\nabla f(\x_t;\;\xi_t)}{\norm{\nabla f(\x_t;\;\xi_t)}},\quad\mbox{ and }\quad \bZ_t^\perp,
\end{equation}
where $\bZ_t^\perp \in\RR^{d\times (d-1)}$ is an orthonormal matrix and orthogonal to $\z_t$. 
It is easy to check that
\begin{equation}\label{eq:decom}
\bm{I}_d 
= \z_t\z_t^\top + \left(\bm{I}_d - \z_t\z_t^\top\right) 
= \z_t\z_t^\top +  \bZ_t^\perp [\bZ_t^\perp]^\top.
\end{equation}

Next, we will introduce the index $\{i_k^t\}$ with $k = 1,\dots, N$ which satisfies the following property
\begin{equation}\label{eq:idx}
	\dotprod{\z_t, \ub_{i_1^t}} 
	\leq 
	\dotprod{\z_t, \ub_{i_2^t}} 
	\leq
	\dots
	\leq
	\dotprod{\z_t, \ub_{i_N^t}}. 
\end{equation}
These notations will be widely used in this section.

Next, we will define a list of events which hold with high probabilities. 
Conditioned on these events, we can prove that the direction $\db_t$ in Eq.~\eqref{eq:dt} can achieve a sufficient value decay for each iteration.
\subsection{Events}
Assume the function $f(\x;\;\xi)$ is $L$-smooth. 
Given $0<\delta<1$, we will define several events as follows:
\begin{align}
	&\cE_{t,1} := \left\{ |D_\xi(\x_t + \alpha \ub_{(k)},\;\x_t)| \leq \Cd L\alpha^2 \mid k = 1,\dots, \frac{N}{4}, \frac{3N}{4}+1, \dots, N \right\} \label{eq:E1}\\
	&\cE_{t,2} := \left\{ \norm{\db_t}^2 \leq \frac{8\Cd}{N}   \right\} \label{eq:E2}\\
	&\cE_{t,3} := \left\{\sum_{k=1}^{N/4}  \dotprod{\z_t, \ub_{i_k}} \leq	\sum_{k=1}^{N/4} \dotprod{\z_t, \ub_{(k)}} \leq \sum_{k=1}^{N/4}  \dotprod{\z_t, \ub_{i_k}} 
	+ 
	\frac{N \Cd L \alpha}{2\norm{\nabla f(\x_t;\;\xi_t)}}\right\} \label{eq:up1}\\
	&\cE_{t,4} := \left\{\sum_{k=3N/4+1}^{N}  \dotprod{\z_t, \ub_{i_k}} \geq	\sum_{k=3N/4+1}^{N} \dotprod{\z_t, \ub_{(k)}} \geq \sum_{k=3N/4+1}^{N}  \dotprod{\z_t, \ub_{i_k}} 
	- 
	\frac{N \Cd L \alpha}{2\norm{\nabla f(\x_t;\;\xi_t)}}\right\} \label{eq:up2}\\
	&\cE_{t,5} := \left\{ \dotprod{\z_t, \ub_{i_{3N/4+1}}} \geq 2,\quad\mbox{and}\quad \dotprod{\z_t, \ub_{i_{N/4}}} \leq -2 \right\}, \label{eq:E5}
\end{align} 
where $D_\xi(\cdot,\cdot)$, and $\Cd$ are defined in Eq.~\eqref{eq:dyx} and Eq.~\eqref{eq:d_up},  respectively.

Next, we will bound the probabilities of the above events happening.

\begin{lemma}\label{lem:E1}
	Assume that function $f(\x;\;\xi)$ is $L$-smooth.
	Given $0<\delta<\frac{2}{N}$, we have
	\begin{equation}\label{eq:d_up}
		\Pr(\cE_{t,1}) \geq 1 - \frac{N\delta}{2}, \mbox{ with } \Cd \triangleq d+2\log\frac{1}{\delta},
	\end{equation}
	with $D_\xi(\y,\x)$ defined as 
	\begin{equation}\label{eq:dyx}
		D_\xi(\y,\x) := f(\y;\;\xi) - f(\x;\;\xi) - \dotprod{\nabla f(\x;\;\xi), \y -\x}.
	\end{equation}
\end{lemma}
\begin{proof}
	Using $\y = \x+\alpha \ub$ in Eq.~\eqref{eq:dyx}, we can obtain that
	\begin{equation}\label{eq:taylor}
		D_\xi(\x + \alpha \ub, \x) = f(\x + \alpha \ub;\;\xi) - f(\x;\;\xi) - \alpha \dotprod{\nabla f(\x;\;\xi), \ub}, 
	\end{equation}
	By the $L$-smoothness of $f(\x;\;\xi)$, we can obtain that
	\begin{equation*}
		|D_\xi(\x + \alpha \ub, \x)|
		\leq 
		\frac{L \alpha^2}{2} \norm{\ub}^2 \stackrel{\eqref{eq:u_norm}}{\leq} \frac{(2d + 3\log\frac{1}{\delta}) L\alpha^2}{2} \leq (d+2\log\frac{1}{\delta}) L \alpha^2 = \Cd L \alpha^2.
	\end{equation*}
	
	Accordingly, we can obtain that $|D_\xi(\x_t+\alpha \ub, \x_t)| \leq \Cd L\alpha^2$ holds for a random vector $\ub$ with a probability at least $1-\delta$. 
	In event $\cE_{t,1}$, there are $N/2$ independent above events. 
	By the probability union bound, we can obtain the result.    
\end{proof}

\begin{lemma}\label{lem:d_norm}
	Given $0<\delta <1$,  then event $\cE_{t,2}$ defined in Eq.~\eqref{eq:E2} holds with a probability at least $1-\delta$.
\end{lemma}
\begin{proof}
	By the definition of $\db_t$ in Eq.~\eqref{eq:dt} and the fact that $|w_{(k)}| = \frac{4}{N}$, we can obtain  
	\begin{align*}
		\norm{\db_t}^2 
		= \frac{16}{N^2}\norm{ \sum_{k\in \mathcal{K}}  \ub_{(k)} }^2.
	\end{align*}
	Furthermore, $\sum_{k\in \mathcal{K}}  \ub_{(k)}$ is Gaussian distribution with parameter $(\bm{0}, \frac{N}{2} \cdot \bm{I}_d)$. 
	Thus, by Eq.~\eqref{eq:u_norm}, we can obtain that
	\begin{equation*}
		\norm{ \sum_{k\in \mathcal{K}}  \ub_{(k)} }^2
		\leq 
		\frac{N \Cd}{2}.
	\end{equation*}
	Combining above results, we can obtain the final result.
\end{proof}

\begin{lemma}\label{lem:E3}
Letting index $(k)$ satisfy Eq.~\eqref{eq:sort} and event $\cE_{t,1}$ hold, then event $\cE_{t,3}$ defined in Eq.~\eqref{eq:up1} holds.
\end{lemma}
\begin{proof}
Since $0<\alpha$ and $\norm{\nabla f(\x_t;\;\xi_t)}$, then  Eq.~\eqref{eq:sort} is equivalent to the following equation
\begin{equation*}
	\frac{f(\x_t+\alpha \ub_{(1)};\;\xi_t) - f(\x_t;\;\xi_t)}{\alpha \norm{\nabla f(\x_t;\;\xi_t)}} \leq 	\frac{f(\x_t+\alpha \ub_{(2)};\;\xi_t) - f(\x_t;\;\xi_t)}{\alpha \norm{\nabla f(\x_t;\;\xi_t)}}\leq\dots\leq \frac{f(\x_t+\alpha \ub_{(N)};\;\xi_t) - f(\x_t;\;\xi_t)}{\alpha \norm{\nabla f(\x_t;\;\xi_t)}}.
\end{equation*}
Combining with Eq.~\eqref{eq:taylor}, we can obtain 
\begin{align*}
	\dotprod{\ub_{(1)}, \z_t} + \frac{D_\xi(\x_t+\alpha \ub_{(1)}, \x_t)}{\alpha \norm{\nabla f(\x_t;\;\xi_t)}}
	\leq 
	\dotprod{\ub_{(2)}, \z_t} + \frac{D_\xi(\x_t+\alpha \ub_{(2)}, \x_t)}{\alpha \norm{\nabla f(\x_t;\;\xi_t)}}
	\leq
	\dots
	\leq 
	\dotprod{\ub_{(N)}, \z_t} + \frac{D_\xi(\x_t+\alpha \ub_{(N)}, \x_t)}{\alpha \norm{\nabla f(\x_t;\;\xi_t)}}.
\end{align*}
For notation convenience, we use an index $\{j_k\}$ to replace $(k)$. 
Accordingly, the index $\{j_k\}$ satisfies
\begin{align}\label{eq:jdx}
	\dotprod{\z_t, \ub_{j_1}} + \frac{D_\xi(\x_t + \alpha \ub_{j_1}, \x_t)}{\alpha \norm{\nabla f(\x_t;\;\xi_t)}}
	\leq 
	\dotprod{\z, \ub_{j_2}} + \frac{D(\x + \alpha \ub_{j_2}, \x)}{\alpha \norm{\nabla f(\x_t;\;\xi_t)}}
	\leq
	\dots
	\leq
	\dotprod{\z, \ub_{j_N}} + \frac{D(\x + \alpha \ub_{j_N}, \x)}{\alpha \norm{\nabla f(\x_t;\;\xi_t)}}.
\end{align} 

Let us denote that
\begin{equation*}
	\mathcal{J} := \{j_1, j_2, \dots, j_{N/4}\}, \quad\mbox{ and } \quad \mathcal{I}:= \{i_1, i_2,\dots,i_{N/4}\},
\end{equation*}
where index $\{i_{k}\}$ satisfies Eq.~\eqref{eq:idx}.

If two sets $\mathcal{J}$ and $\mathcal{I}$ are the same,  then it holds that
\begin{equation*}
	\sum_{k=1}^{N/4} \dotprod{\z, \ub_{j_k}} = \sum_{k=1}^{N/4} \dotprod{\z, \ub_{i_k}}.
\end{equation*}
Then event $\cE_{t,3}$ holds trivially.

Otherwise, the set that $\mathcal{J}$ minuses $\mathcal{I}$ which is denoted as $\mathcal{J} \setminus \mathcal{I}$ is non-empty. 
Assume that  $\mathcal{J} \setminus \mathcal{I}$ has $\ell$ elements. 
Accordingly, the set $\mathcal{I} \setminus \mathcal{J}$ also has $\ell$ elements.
We can present $\mathcal{J} \setminus \mathcal{I}$ as follows:
\begin{equation*}
	\mathcal{J} \setminus \mathcal{I} = \{ i_{N/4 + k_1}, i_{N/4 + k_2}, \dots, i_{N/4 + k_\ell}  \}, \mbox{ with } 1\leq k_1 \leq k_2\leq\dots\leq k_\ell.
\end{equation*}
This is because if $k_i\leq 0$ with $1\leq i\leq \ell$, then $i_{N/4 + k_i}$  will be in $\mathcal{I}$. 

Similarly, the set $\mathcal{I} \setminus \mathcal{J}$ can be presented as:
\begin{align*}
	\mathcal{I} \setminus \mathcal{J} = \{i_{k_1'}, i_{k_2'},\dots, i_{k_\ell'}\} \mbox{ with } 1\leq k_1' \leq k_2'\leq \dots\leq k_\ell' \leq N/4.
\end{align*}

Then for any $k_i$ and $k_i'$ with $i=1,\dots,\ell$, we have
\begin{align*}
	\dotprod{\z_t, \ub_{i_{N/4+k_i}}} + \frac{D_\xi(\x_t + \alpha \ub_{i_{N/4+k_i}}, \x_t)}{\alpha \norm{\nabla f(\x_t;\;\xi_t)}}
	\leq 
	\dotprod{\z_t, \ub_{i_{k_i'}}} + \frac{D_\xi(\x_t + \alpha \ub_{i_{k_i'}}, \x_t)}{\alpha \norm{\nabla f(\x_t;\;\xi_t)}}.
\end{align*}
This is because of $i_{N/4 + k_1} \in \mathcal{J}$, $i_{k_i'} \notin \mathcal{J}$, and Eq.~\eqref{eq:jdx}. 

Thus, we have
\begin{align*}
	\dotprod{\z_t, \ub_{i_{N/4+k_i}}} 
	\leq&
	\dotprod{\z_t, \ub_{i_{k_i'}}}
	+
	\frac{D_\xi(\x_t + \alpha \ub_{i_{k_i'}}, \x_t)}{\alpha \norm{\nabla f(\x_t;\;\xi_t)}} 
	-
	\frac{D_\xi(\x_t + \alpha \ub_{i_{N/4+k_i}}, \x_t)}{\alpha \norm{\nabla f(\x_t;\;\xi_t)}}\\
	\stackrel{\eqref{eq:E1}}{\leq}&
	\dotprod{\z_t, \ub_{i_{k_i'}}}
	+ 
	\frac{2\Cd L \alpha}{\norm{\nabla f(\x_t;\;\xi_t)}},
\end{align*}
where the last inequality is because event $\cE_{t,1}$ holds.

Furthermore, it also holds that
\begin{equation*}
	\dotprod{\z_t, \ub_{i_{k_i'}}} \leq \dotprod{\z_t, \ub_{i_{N/4+k_i}}}.
\end{equation*}
This is because of Eq.~\eqref{eq:idx} and $k_i' \leq N/4+k_i$. 
Thus, it holds that
\begin{equation}\label{eq:jb}
	\dotprod{\z_t, \ub_{i_{k_i'}}} 
	\leq 
	\dotprod{\z_t, \ub_{i_{N/4+k_i}}}
	\leq 
	\dotprod{\z_t, \ub_{i_{k_i'}}}
	+ 
	\frac{2\Cd L \alpha}{\norm{\nabla f(\x_t;\;\xi_t)}}.
\end{equation}
Accordingly, we can obtain that
\begin{align*}
	\sum_{i\in \mathcal{I}\setminus\mathcal{J}} \dotprod{\z_t, \ub_i} 
	\leq 	
	\sum_{j\in \mathcal{J}\setminus\mathcal{I}} \dotprod{\z_t, \ub_j}
	\leq
	\sum_{i\in \mathcal{I}\setminus\mathcal{J}} \dotprod{\z_t, \ub_i} 
	+ 
	\frac{2\ell \Cd L \alpha}{\norm{\nabla f(\x_t;\;\xi_t)}}.
\end{align*}

Furthermore, it holds that
\begin{align*}
	\mathcal{J} =& (\mathcal{J}\setminus\mathcal{I}) \;\cup\; (\mathcal{J}\cap \mathcal{I}  ), \mbox{ and } (\mathcal{J}\setminus\mathcal{I}) \;\cap\; (\mathcal{J}\cap \mathcal{I}  ) = \emptyset,\\
	\mathcal{I} =& (\mathcal{I}\setminus\mathcal{J}) \;\cup\; (\mathcal{J}\cap \mathcal{I}  ), \mbox{ and } (\mathcal{I}\setminus\mathcal{J}) \;\cap\; (\mathcal{J}\cap \mathcal{I}  ) = \emptyset.
\end{align*}
Thus, we have
\begin{align*}
	\sum_{j\in\mathcal{J}}  \dotprod{\z_t, \ub_j} 
	=& 	
	\sum_{j\in \mathcal{J}\setminus\mathcal{I}} \dotprod{\z_t, \ub_j}
	+
	\sum_{j\in \mathcal{J}\cap\mathcal{I}} \dotprod{\z_t, \ub_j}\\
	=&
	\sum_{j\in \mathcal{J}\setminus\mathcal{I}} \dotprod{\z_t, \ub_j}
	+
	\sum_{i\in \mathcal{J}\cap\mathcal{I}} \dotprod{\z_t, \ub_i}\\
	\stackrel{\eqref{eq:jb}}{\leq}&
	\sum_{i\in \mathcal{I}\setminus\mathcal{J}} \dotprod{\z_t, \ub_i} 
	+ 
	\frac{2\ell \Cd L \alpha}{\norm{\nabla f(\x_t;\;\xi_t)}}
	+
	\sum_{i\in \mathcal{J}\cap\mathcal{I}} \dotprod{\z_t, \ub_i}\\
	=&
	\sum_{i\in \mathcal{I}} \dotprod{\z_t, \ub_i} 
	+ 
	\frac{2\ell \Cd L \alpha}{\norm{\nabla f(\x_t;\;\xi_t)}}\\
	\leq&
	\sum_{i\in \mathcal{I}} \dotprod{\z_t, \ub_i} 
	+ 
	\frac{N \Cd L \alpha}{2\norm{\nabla f(\x_t;\;\xi_t)}}.
\end{align*}
Similarly, we can obtain that
\begin{align*}
	\sum_{i\in \mathcal{I}} \dotprod{\z_t, \ub_i} 
	\leq 
	\sum_{j\in\mathcal{J}}  \dotprod{\z_t, \ub_j}.
\end{align*}

Thus, we have
\begin{equation*}
	\sum_{i\in \mathcal{I}} \dotprod{\z_t, \ub_i} 
	\leq 
	\sum_{j\in\mathcal{J}}  \dotprod{\z_t, \ub_j}
	\leq
	\sum_{i\in \mathcal{I}} \dotprod{\z_t, \ub_i}
	+ 
	\frac{N \Cd L \alpha}{2\norm{\nabla f(\x_t;\;\xi_t)}}.
\end{equation*}
\end{proof}

\begin{lemma}\label{lem:E4}
	Letting index $(k)$ satisfy Eq.~\eqref{eq:sort} and event $\cE_{t,1}$ hold, then event $\cE_{t,4}$ defined in Eq.~\eqref{eq:up2} holds.
\end{lemma}
\begin{proof}
The proof is almost the same as the one of Lemma~\ref{lem:E3}.
\end{proof}

\begin{lemma}\label{lem:E5}
	Event $\cE_{t,5}$ defined in Eq.~\eqref{eq:E5} holds with a probability at least $1 - 2\exp(-N \cdot K(1/4 \Vert p))$ with $p = 0.0224$. That is,
	\begin{equation*}
		\Pr\left(\cE_{t,5}\right) \geq 1 - 2\exp(-N \cdot K(1/4 \Vert p)), \quad\mbox{ with }\quad p = 0.0224,
	\end{equation*}
	where the binary Kullback--Leibler divergence is defined as
	\begin{equation*}
		K(q \Vert p) \;=\; q\ln\frac{q}{p}+(1-q)\ln\frac{1-q}{1-p}.
	\end{equation*}
\end{lemma}
\begin{proof}
	First, it holds that $\dotprod{\z_t, \ub_i} \sim \cN(0,1)$. 
	Thus, $\dotprod{\z_t, \ub_1},\dots, \dotprod{\z_t, \ub_N}$ i.i.d. follow the standard Gaussian distribution.  
	We denote them $X_1,\dots, X_N$. 
	Accordingly $X_{i_k}$ is the order statistics of $X_1,\dots, X_N$.
	Then, by Lemma~\ref{lem:low1} and Lemma~\ref{lem:low2}, we can obtain the result in Eq.~\eqref{eq:E5}.

\end{proof}

\section{Descent Direction}

Next, we will prove that $\db_t$ is a real descent direction. 
We first represent the inner product of $\db_t$ and $\nabla f(\x_t)$ in a correct form.

\begin{lemma}\label{lem:dd}
Let $\db_t^+$ and $\db_t^-$ be defined in Eq.~\eqref{eq:dt} with all $w_{(k)}^+ = \frac{4}{N}$ and $w_{(k)}^- = -\frac{4}{N}$. 
Notations $\z_t$ and $\bZ_t^\perp$ are defined in Eq.~\eqref{eq:z}.
Then, we have the following properties:
\begin{equation}
\dotprod{\db_t^+, \nabla f(\x_t)} 
=
\frac{4\dotprod{\nabla f(\x_t), \z_t}}{N} \sum_{k=1}^{N/4}  \z_t^\top \ub_{(k)} 
+  
\frac{4 \nabla^\top f(\x_t) \bZ_t^\perp}{N} \sum_{k=1}^{N/4}  [\bZ_t^\perp]^\top \ub_{(k)},
\end{equation}
and
\begin{equation}
	\dotprod{\db_t^-, \nabla f(\x_t)} 
	=
	-\frac{4\dotprod{\nabla f(\x_t), \z_t}}{N} \sum_{k=3N/4+1}^{N}  \z_t^\top \ub_{(k)} 
	-  
	\frac{4 \nabla^\top f(\x_t) \bZ_t^\perp}{N} \sum_{k=3N/4+1}^{N}  [\bZ_t^\perp]^\top \ub_{(k)}.
\end{equation}
Furthermore, it holds that $\z_t^\top \ub_{(k)} \sim \cN(0, 1)$ and all $[\bZ_t^\perp]^\top \ub_{(k)}$ i.i.d. follow the $(d-1)$-dimensional standard Gaussian, that is  $[\bZ_t^\perp]^\top \ub_{(k)} \sim \cN(\bm{0}, \bm{I}_{d-1})$. 
Furthermore,
$\z_t^\top \ub_{(k)}$ and $[\bZ_t^\perp]^\top \ub_{(k)}$ are independent to each other. 
\end{lemma}
\begin{proof}
	First, by the definition of $\db_t^+$ in Eq.~\eqref{eq:dt}, we have
\begin{align*}
\dotprod{\db_t^+, \nabla f(\x_t)}
= 
\frac{4}{N} \sum_{k=1}^{N/4} \dotprod{\ub_{(k)}, \nabla f(\xb_t)  }.
\end{align*}

Next, we will decompose $\ub_{(k)}$ as follows by Eq.~\eqref{eq:decom}:
\begin{equation}\label{eq:decomp}
\ub_{(k)} = \z_t\z_t^\top \ub_{(k)} + \bZ_t^\perp[\bZ_t^\perp]^\top \ub_{(k)}.
\end{equation}
By the properties of the Gaussian distribution, we can obtain that
$\z_t^\top \ub_{(k)} \sim \cN(0, 1)$ and $[\bZ_t^\perp]^\top \ub_{(k)} \sim \cN(0, \bm{I}_{d-1})$. 
Furthermore, $\z_t^\top \ub_{(k)}$ and $[\bZ_t^\perp]^\top \ub_{(k)}$ are independent to each other.

Using the decomposition in Eq.~\eqref{eq:decomp}, we can obtain that 
\begin{align*}
\dotprod{\db_t^+, \nabla f(\x_t)} 
=
\frac{4\dotprod{\nabla f(\x_t), \z_t}}{N} \sum_{k=1}^{N/4}  \z_t^\top \ub_{(k)} 
+  
\frac{4 \nabla^\top f(\x_t) \bZ_t^\perp}{N} \sum_{k=1}^{N/4}  [\bZ_t^\perp]^\top \ub_{(k)}.
\end{align*}

Similarly,
\begin{align*}
\dotprod{\db_t^-, \nabla f(\x_t)} 
=
-\frac{4\dotprod{\nabla f(\x_t), \z_t}}{N} \sum_{k=3N/4+1}^{N}  \z_t^\top \ub_{(k)} 
-  
\frac{4 \nabla^\top f(\x_t) \bZ_t^\perp}{N} \sum_{k=3N/4+1}^{N}  [\bZ_t^\perp]^\top \ub_{(k)}.
\end{align*}
\end{proof}

Based on Lemma~\ref{lem:dd}, we will present the inner product $\dotprod{\db_t, \nabla f(\x_t)}$ in the following form.
\begin{lemma}\label{lem:d_nab}
Letting $\db_t$ be defined in Eq.~\eqref{eq:dt}, then we have
\begin{equation}\label{eq:d_nab}
\begin{aligned}
\dotprod{\db_t, \nabla f(\x_t)}
=&
\frac{4 \norm{\nabla f(\x_t)}^2}{N \norm{\nabla f(\x_t;\;\xi_t)}}\left(\sum_{k=1}^{N/4}  \z_t^\top \ub_{(k)} - \sum_{k=3N/4+1}^{N} \z_t^\top \ub_{(k)} \right)
+ \frac{4 \nabla^\top f(\x_t) \bZ_t^\perp}{N} \sum_{k\in\cK}\bb_{(k)}^t\\
&+ 
\frac{4\dotprod{\nabla f(\x_t),  \nabla f(\x_t;\;\xi_t) - \nabla f(\x_t)}}{N\norm{\nabla f(\x_t;\;\xi_t)}}\left(\sum_{k=1}^{N/4}  \z_t^\top \ub_{(k)} - \sum_{k=3N/4+1}^{N} \z_t^\top \ub_{(k)} \right),
\end{aligned}
\end{equation}
where we define $\bb_{(k)}^t = [\bZ_t^\perp]^\top \ub_{(k)}$.
\end{lemma}
\begin{proof}
First, by the definition of $\z_t$ in Eq.~\eqref{eq:z}, we have
\begin{equation}\label{eq:nab_z_decom}
\dotprod{\nabla f(\x_t), \z_t}
=
\frac{\dotprod{\nabla f(\x_t), \nabla f(\x_t;\;\xi_t)}}{\norm{\nabla f(\x_t;\;\xi_t)}}
=
\frac{\dotprod{\nabla f(\x_t), \nabla f(\x_t)}}{\norm{\nabla f(\x_t;\;\xi_t)}}
+
\frac{\dotprod{\nabla f(\x_t),  \nabla f(\x_t;\;\xi_t) - \nabla f(\x_t)}}{\norm{\nabla f(\x_t;\;\xi_t)}}.
\end{equation}
Then, we can obtain that
\begin{align*}
	\dotprod{\db_t, \nabla f(\x_t)}
	\stackrel{\eqref{eq:dt}}{=}&
	\dotprod{\db_t^+, \nabla f(\x_t)}
	+
	\dotprod{\db_t^-, \nabla f(\x_t)}\\
	=&
	\frac{4\dotprod{\nabla f(\x_t), \z_t}}{N} \left(\sum_{k=1}^{N/4}  \z_t^\top \ub_{(k)} - \sum_{k=3N/4+1}^{N} \z_t^\top \ub_{(k)} \right) 
	+ 
	\frac{4 \nabla^\top f(\x_t) \bZ_t^\perp}{N} \sum_{k\in\cK}\bb_{(k)}^t\\
	\stackrel{\eqref{eq:nab_z_decom}}{=}&
	\frac{4 \norm{\nabla f(\x_t)}^2}{N \norm{\nabla f(\x_t;\;\xi_t)}}\left(\sum_{k=1}^{N/4}  \z_t^\top \ub_{(k)} - \sum_{k=3N/4+1}^{N} \z_t^\top \ub_{(k)} \right)
	+ \frac{4 \nabla^\top f(\x_t) \bZ_t^\perp}{N} \sum_{k\in\cK}\bb_{(k)}^t\\
	&+ 
	\frac{4\dotprod{\nabla f(\x_t),  \nabla f(\x_t;\;\xi_t) - \nabla f(\x_t)}}{N\norm{\nabla f(\x_t;\;\xi_t)}}\left(\sum_{k=1}^{N/4}  \z_t^\top \ub_{(k)} - \sum_{k=3N/4+1}^{N} \z_t^\top \ub_{(k)} \right),
\end{align*}
where the second equality is because of Lemma~\ref{lem:dd}.
\end{proof}

Next, we will bound the value of $\sum_{k=1}^{N/4}  \z_t^\top \ub_{(k)} - \sum_{k=3N/4+1}^{N} \z_t^\top \ub_{(k)}$. 
Based on events  $\cE_{t,3}$-$\cE_{t,5}$, we can obtain the following result.
\begin{lemma}
Assuming that events $\cE_{t,3}$-$\cE_{t,5}$ hold, then we have the following property
\begin{equation}\label{eq:aas}
\sum_{k=1}^{N/4}  \z_t^\top \ub_{(k)} - \sum_{k=3N/4+1}^{N} \z_t^\top \ub_{(k)}
\leq 
-
N
+ 
\frac{N \Cd L \alpha}{\norm{\nabla f(\x_t;\;\xi_t)}}.
\end{equation}
\end{lemma}
\begin{proof}
	We have
\begin{align*}
	&
	\sum_{k=1}^{N/4} \dotprod{\z_t, \ub_{(k)}} - \sum_{k=3N/4+1}^{N}\dotprod{\z_t, \ub_{(k)}}\\
	\stackrel{\eqref{eq:up1}\eqref{eq:up2}}{\le}&
	\sum_{k=1}^{N/4} \dotprod{\z_t, \ub_{i_k}} + \frac{N \Cd L \alpha}{2\norm{\nabla f(\x_t;\;\xi_t)}} 
	- \sum_{k=3N/4+1}^{N} \dotprod{\z_t, \ub_{i_k}} + \frac{N \Cd L \alpha}{2\norm{\nabla f(\x_t;\;\xi_t)}}\\
	\stackrel{\eqref{eq:idx}}{\leq}&
	\frac{N}{4} \dotprod{\z_t, \ub_{i_{N/4}}}  
	- 
	\frac{N}{4} \dotprod{\z_t, \ub_{i_{3N/4+1}}} 
	+
	\frac{N \Cd L \alpha}{\norm{\nabla f(\x_t;\;\xi_t)}} \\
	\stackrel{\eqref{eq:E5}}{\leq}&
	-
	N
	+ 
	\frac{N \Cd L \alpha}{\norm{\nabla f(\x_t;\;\xi_t)}},
\end{align*}
where the last inequality is because if $\dotprod{\z_t, \ub_{3N/4+1}} \geq 2$, and $\dotprod{\z_t, \ub_{N/4}} \leq -2$ shown in Eq.~\eqref{eq:E5}. 
\end{proof}

Based on the above lemmas, we will show  how the function value decays after one-step updates if the smooth parameter $\alpha$ is properly chosen. 
\begin{lemma}\label{lem:dec}
Assume that $f(\x)$ and $f(\x;\;\xi)$ are $L$-smooth, and events $\cE_{t,1}$ - $\cE_{t,5}$ hold.
If the smooth parameter $\alpha$ satisfies $0< \alpha \leq \frac{\norm{\nabla f(\x_t;\;\xi_t)}}{2\Cd L}$, then Algorithm~\ref{alg:SA1} has the following descent property
	\begin{equation}\label{eq:dec}
	\begin{aligned}
	f(\x_{t+1}) 
	\leq&
	f(\x_t) 
	- 
	2\eta_t \frac{\norm{\nabla f(\x_t)}^2}{\norm{\nabla f(\x_t;\;\xi_t)}}  
	+ 
	\eta_t\frac{4 \nabla^\top f(\x_t) \bZ_t^\perp}{N} \sum_{k\in\cK}\bb_{(k)}^t\\
	&+ 
	\eta_t\frac{4\dotprod{\nabla f(\x_t),  \nabla f(\x_t;\;\xi_t) - \nabla f(\x_t)}}{N\norm{\nabla f(\x_t;\;\xi_t)}}\left(\sum_{k=1}^{N/4}  \z_t^\top \ub_{(k)} - \sum_{k=3N/4+1}^{N} \z_t^\top \ub_{(k)} \right)\\
	&+ 
	\frac{4 \eta_t^2 L \Cd}{N}.
	\end{aligned}
	\end{equation}
\end{lemma}
\begin{proof}
By the $L$-smoothness of $f(\x)$, we can obtain that
\begin{align*}
f(\x_{t+1}) 
\stackrel{\eqref{eq:L}\eqref{eq:update}}{\leq}& 
f(\x_t) + \eta_t \dotprod{\nabla f(\x_t), \db_t} + \frac{\eta_t^2 L}{2} \norm{\db_t}^2\\
\stackrel{\eqref{eq:d_nab}\eqref{eq:aas}}{\leq}&
f(\x_t) 
+ 
4\eta_t  \left( -1 + \frac{\Cd L\alpha}{\norm{\nabla f(\x_t;\;\xi_t)}}\right) \frac{\norm{\nabla f(\x_t)}^2}{\norm{\nabla f(\x_t;\;\xi_t)}}  
+ 
\eta_t\frac{4 \nabla^\top f(\x_t) \bZ_t^\perp}{N} \sum_{k\in\cK}\bb_{(k)}^t\\
&+ 
\eta_t\frac{4\dotprod{\nabla f(\x_t),  \nabla f(\x_t;\;\xi_t) - \nabla f(\x_t)}}{N\norm{\nabla f(\x_t;\;\xi_t)}}\left(\sum_{k=1}^{N/4}  \z_t^\top \ub_{(k)} - \sum_{k=3N/4+1}^{N} \z_t^\top \ub_{(k)} \right)
+ 
\frac{\eta_t^2 L}{2} \norm{\db_t}^2\\
\stackrel{\eqref{eq:E2}}{\leq}&
f(\x_t) 
+ 
4\eta_t  \left( -1 + \frac{\Cd L\alpha}{\norm{\nabla f(\x_t;\;\xi_t)}}\right) \frac{\norm{\nabla f(\x_t)}^2}{\norm{\nabla f(\x_t;\;\xi_t)}}  
+ 
\eta_t\frac{4 \nabla^\top f(\x_t) \bZ_t^\perp}{N} \sum_{k\in\cK}\bb_{(k)}^t\\
&+ 
\eta_t\frac{4\dotprod{\nabla f(\x_t),  \nabla f(\x_t;\;\xi_t) - \nabla f(\x_t)}}{N\norm{\nabla f(\x_t;\;\xi_t)}}\left(\sum_{k=1}^{N/4}  \z_t^\top \ub_{(k)} - \sum_{k=3N/4+1}^{N} \z_t^\top \ub_{(k)} \right)
+ 
\frac{4 \eta_t^2 L \Cd}{N}\\
\leq&
f(\x_t) 
- 
2\eta_t \frac{\norm{\nabla f(\x_t)}^2}{\norm{\nabla f(\x_t;\;\xi_t)}}  
+ 
\eta_t\frac{4 \nabla^\top f(\x_t) \bZ_t^\perp}{N} \sum_{k\in\cK}\bb_{(k)}^t\\
&+ 
\eta_t\frac{4\dotprod{\nabla f(\x_t),  \nabla f(\x_t;\;\xi_t) - \nabla f(\x_t)}}{N\norm{\nabla f(\x_t;\;\xi_t)}}\left(\sum_{k=1}^{N/4}  \z_t^\top \ub_{(k)} - \sum_{k=3N/4+1}^{N} \z_t^\top \ub_{(k)} \right)
+ 
\frac{4 \eta_t^2 L \Cd}{N},
\end{align*}
where the last inequality is because of $\alpha \leq \frac{\norm{\nabla f(\x_t;\;\xi_t)}}{2\Cd L}$.
\end{proof}

Next, we will define the natural filtration as the $\sigma$-algebra generated by the random variables $\xi_i$'s  observed to step $t$: 
\begin{equation}\label{eq:cF}
	\cF_t:= \sigma(\x_1,\xi_1,\xi_2,\dots, \xi_{t-1}).
\end{equation}

\begin{lemma}\label{lem:dec3}
Assume that $f(\x)$ and $f(\x;\;\xi)$ are $L$-smooth, and events $\cE_{t,1}$ - $\cE_{t,5}$ hold.
Let the smooth parameter $\alpha$ satisfy $0< \alpha \leq \frac{\norm{\nabla f(\x_t;\;\xi_t)}}{2\Cd L}$.
Let the stochastic gradient $\nabla f(\x_t;\;\xi)$ be of bounded second moment, that is, Assumption~\ref{ass:bsm} holds. 
Setting the step size $\eta_t = \heta_t \norm{\nabla f(\x_t;\;\xi_t)}$ with $\heta_t$ independent of $\xi_t$, then Algorithm~\ref{alg:SA1} has the following descent property
\begin{equation}\label{eq:dec3}
\begin{aligned}
	\EE\left[ f(\x_{t+1}) \mid \cF_t \right]
	\leq
	f(\x_t) 
	-
	2\heta_t \norm{\nabla f(\x_t)}^2
	+
	\frac{4\heta_t \nabla^\top  f(\x_t) \EE\left[\tbZ_{\xi_t}^\perp \mid \cF_t\right]}{N} \sum_{k\in\cK}\bb_{(k)}^t
	+
	\frac{4 \heta_t^2 L \Cd G_u^2}{N},
\end{aligned}
\end{equation}
where $\cF_t$ is defined in Eq.~\eqref{eq:cF} and we define
\begin{equation}\label{eq:tbz}
	\tbZ_{\xi_t}^\perp := \norm{\nabla f(\x_t;\;\xi_t)} 
	\cdot \bm{I}_d - \frac{\nabla f(\x_t;\;\xi_t) \nabla^\top f(\x_t;\;\xi_t)}{\norm{\nabla f(\x_t;\;\xi_t)}}.
\end{equation}
\end{lemma}
\begin{proof}
	First, by Eq.~\eqref{eq:dec} and the step size $\eta_t = \heta_t \norm{\nabla f(\x_t;\;\xi_t)}$, we have 
	\begin{align*}
		f(\x_{t+1}) 
		\stackrel{\eqref{eq:dec}}{\leq}&
		f(\x_t)  
		- 
		2\eta_t \frac{\norm{\nabla f(\x_t)}^2}{\norm{\nabla f(\x_t;\;\xi_t)}}  
		+ 
		\eta_t\frac{4 \nabla^\top f(\x_t) \bZ_t^\perp}{N} \sum_{k\in\cK}\bb_{(k)}^t\\
		&+ 
		\frac{4\eta_t\dotprod{\nabla f(\x_t),  \nabla f(\x_t;\;\xi_t) - \nabla f(\x_t)}}{N\norm{\nabla f(\x_t;\;\xi_t)}}\left(\sum_{k=1}^{N/4}  \z_t^\top \ub_{(k)} - \sum_{k=3N/4+1}^{N} \z_t^\top \ub_{(k)} \right)
		+ 
		\frac{4 \eta_t^2 L \Cd}{N}\\
		=&
		f(\x_t)  
		-
		2 \heta_t  \norm{\nabla f(\x_t)}^2 
		+ 
		\frac{4\heta_t \norm{\nabla f(\x_t;\;\xi_t)} \nabla^\top f(\x_t) \bZ_t^\perp}{N} \sum_{k\in\cK}\bb_{(k)}^t\\
		&
		+ 
		\frac{4\heta_t \dotprod{\nabla f(\x_t),  \nabla f(\x_t;\;\xi_t) - \nabla f(\x_t)}}{N}
		\left(\sum_{k=1}^{N/4}  \z_t^\top \ub_{(k)} - \sum_{k=3N/4+1}^{N} \z_t^\top \ub_{(k)} \right)\\
		&+ 
		\frac{4 \heta_t^2 L \Cd \norm{\nabla f(\x_t;\;\xi_t)}^2}{N},
	\end{align*}
	where the equality  is because the setting $\eta_t = \heta_t \norm{\nabla f(\x_t;\;\xi_t)}$.
	
	Taking expectation to above equation conditioned on $\cF_t$, we can obtain that
	\begin{align*}
		&\EE\left[ f(\x_{t+1}) \mid \cF_t \right]\\
		\leq&
		f(\x_t) 
		-
		2\heta_t \norm{\nabla f(\x_t)}^2
		+
		\frac{4\heta_t \EE\left[\norm{\nabla f(\x_t;\;\xi_t)} \nabla^\top f(\x_t) \bZ_t^\perp \mid \cF_t\right]}{N} \sum_{k\in\cK}\bb_{(k)}^t\\
		&
		+ 
		\frac{4\heta_t \dotprod{\nabla f(\x_t), \EE\left[ \nabla f(\x_t;\;\xi_t) \mid \cF_t \right] - \nabla f(\x_t)}}{N}
		\left(\sum_{k=1}^{N/4}  \z_t^\top \ub_{(k)} - \sum_{k=3N/4+1}^{N} \z_t^\top \ub_{(k)} \right)\\
		&
		+ 
		\frac{4 \heta_t^2 L \Cd \EE\left[\norm{\nabla f(\x_t;\;\xi_t)}^2 \mid \cF_t\right]}{N}\\
		=&
		f(\x_t) 
		-
		2\heta_t \norm{\nabla f(\x_t)}^2
		+
		\frac{4\heta_t \EE\left[\norm{\nabla f(\x_t;\;\xi_t)} \nabla^\top f(\x_t) \bZ_t^\perp \mid \cF_t\right]}{N} \sum_{k\in\cK}\bb_{(k)}^t\\
		&
		+ 
		\frac{4 \heta_t^2 L \Cd \EE\left[\norm{\nabla f(\x_t;\;\xi_t)}^2 \mid \cF_t\right]}{N}\\
		\leq&
		f(\x_t) 
		-
		2\heta_t \norm{\nabla f(\x_t)}^2
		+
		\frac{4\heta_t \EE\left[\norm{\nabla f(\x_t;\;\xi_t)} \nabla^\top f(\x_t) \bZ_t^\perp \mid \cF_t\right]}{N} \sum_{k\in\cK}\bb_{(k)}^t\\
		&+
		\frac{4 \heta_t^2 L \Cd G_u^2}{N},
	\end{align*} 
	where the equality is because of $\EE\left[ \nabla f(\x_t;\;\xi_t) \mid \cF_t \right] = \nabla f(\x_t)$, and the last inequality is because of the Assumption~\ref{ass:bsm}.
	
	Furthermore, by the definition of $\bZ_t^\perp$, we can obtain that
	\begin{align*}
		&\EE\left[\norm{\nabla f(\x_t;\;\xi_t)} \nabla^\top f(\x_t) \bZ_t^\perp \mid \cF_t\right]\\
		=&
		\EE\left[\norm{\nabla f(\x_t;\;\xi_t)} \nabla^\top f(\x_t) \left(\bm{I}_d - \frac{\nabla f(\x_t;\;\xi_t) \nabla^\top f(\x_t;\;\xi_t)}{\norm{\nabla f(\x_t;\;\xi_t)}^2}\right) \mid \cF_t \right]\\
		=&
		\nabla^\top f(\x_t) \EE
		\left[\norm{\nabla f(\x_t;\;\xi_t)} 
		\cdot \bm{I} - \frac{\nabla f(\x_t;\;\xi_t) \nabla^\top f(\x_t;\;\xi_t)}{\norm{\nabla f(\x_t;\;\xi_t)}} \mid \cF_t \right]\\
		=&
		\nabla^\top f(\x_t) \EE \left[ \tbZ_{\xi_t}^\perp \mid \cF_t \right].
	\end{align*}
	
	Combining the above results, we can obtain that 
	\begin{align*}
		\EE\left[ f(\x_{t+1}) \mid \cF_t \right]
		\leq
		f(\x_t) 
		-
		2\heta_t \norm{\nabla f(\x_t)}^2
		+
		\frac{4\heta_t \nabla^\top  f(\x_t) \EE\left[\tbZ_{\xi_t}^\perp \mid \cF_t\right]}{N} \sum_{k\in\cK}\bb_{(k)}^t
		+
		\frac{4 \heta_t^2 L \Cd G_u^2}{N}.
	\end{align*}

\end{proof}

\begin{lemma}\label{lem:dec1}
	Let  functions $f(\x)$ and $f(\x;\;\xi)$  satisfy the properties in Lemma~\ref{lem:dec3}.
	Furthermore, assume that $f(\x)$ is also $\mu$-strongly convex. 
Assume conditions in Lemma~\ref{lem:dec3} all hold. 
By setting step size $\eta_t = \heta_t\norm{\nabla f(\x_t;\;\xi_t)}$ with $\heta_t = \frac{1}{2\mu t}$,  then it holds that
\begin{equation}\label{eq:dec1}
\begin{aligned}
&\EE\left[f(\x_{t+1}) - f(\x^*) \mid \cF_t\right]\\
\leq&
\left(1 - \frac{2}{t}\right)\Big(f(\x_t) - f(\x^*) \Big) 
+
\frac{2 \nabla^\top f(\x_t) \EE\left[\tbZ_{\xi_t}^\perp \mid \cF_t\right]  }{N\mu t} \sum_{k\in\cK}\bb_{(k)}^t
+ 
\frac{  L \Cd G_u^2}{ N \mu^2 t^2},
\end{aligned}
\end{equation}
where $\x^*$ is optimal point of $f(\x)$.
\end{lemma}
\begin{proof}
First, we represent Eq.~\eqref{eq:dec3} as
\begin{align*}
&\EE\left[ f(\x_{t+1}) - f(\x^*)\mid \cF_t \right]\\
\leq&
f(\x_t) - f(\x^*)
-
2\heta_t \norm{\nabla f(\x_t)}^2
+
\frac{4\heta_t \nabla^\top  f(\x_t) \EE\left[\tbZ_{\xi_t}^\perp \mid \cF_t\right]}{N} \sum_{k\in\cK}\bb_{(k)}^t
+
\frac{4 \heta_t^2 L \Cd G_u^2}{N}\\
\leq&
\Big(1 - 4\mu\heta_t\Big)\cdot \Big(f(\x_t) - f(\x^*)\Big)
+
\frac{4\heta_t \nabla^\top  f(\x_t) \EE\left[\tbZ_{\xi_t}^\perp \mid \cF_t\right]}{N} \sum_{k\in\cK}\bb_{(k)}^t
+
\frac{4 \heta_t^2 L \Cd G_u^2}{N}\\
=&
\left(1 - \frac{2}{t}\right)\cdot \Big(f(\x_t) - f(\x^*)\Big)
+
\frac{2 \nabla^\top  f(\x_t) \EE\left[\tbZ_{\xi_t}^\perp \mid \cF_t\right]}{N\mu t} \sum_{k\in\cK}\bb_{(k)}^t
+
\frac{  L \Cd G_u^2}{N\mu^2t^2},
\end{align*}
where the last inequality is because of Lemma~\ref{lem:str_cvx} and the last equality is because of $\heta_t = \frac{1}{2\mu t}$.
\end{proof}

\begin{lemma}\label{lem:dec_2}
	Let  functions $f(\x)$ and $f(\x;\;\xi)$  satisfy the properties in Lemma~\ref{lem:dec1}.
	All conditions in Lemma~\ref{lem:dec1} also hold. 
	Let $\cE_t = \cap_{i=1,\dots,^5} \cE_{t,i}$ be the event that events $\cE_{t,1}$-$\cE_{t,5}$ all hold. 
	We assume that events $\{\cE_{i}\}$ with $i = 1,\dots, t$ all hold.
Then, it holds that
\begin{equation}\label{eq:dec2}
\begin{aligned}
&\EE\left[f(\x_{t+1}) - f(\x^*)\right]\\
\leq&
\frac{2}{N \mu t(t-1) }\sum_{i=2}^{t} (i-1) \EE\left[ \EE \left[ \nabla^\top f(\x_i) \EE\left[ \bZ_{\xi_i}^\perp \mid \cF_i \right] \mid \cF_{i-1} \right]   \right] \sum_{k\in\cK}\bb_{(k)}^i
+
\frac{  L \Cd G_u^2}{4 N \mu^2 t}.
\end{aligned}
\end{equation}
\end{lemma}
\begin{proof}
Because event $\cE_i$ holds for $i = 1,\dots, t$, Eq.~\eqref{eq:dec1} holds for any iteration $1,\dots, t$.
Thus, combining with the law of total expectation, we unwind Eq.~\eqref{eq:dec1} recursively 
\begin{align*}
&\EE\Big[ f(\x_{t+1}) - f(\x^*) \Big]
=
\EE\Big[\EE\left[ f(\x_{t+1}) - f(\x^*) \mid \cF_t \right]\Big]\\
\stackrel{\eqref{eq:dec1}}{\leq}&
\left(1 - \frac{2}{t}\right)\EE\Big(f(\x_t) - f(\x^*) \Big) 
+
\frac{2 \EE\left[\nabla^\top f(\x_t) \EE\left[\tbZ_{\xi_t}^\perp \mid \cF_t\right]\right]  }{N\mu t} \sum_{k\in\cK}\bb_{(k)}^t
+ 
\frac{  L \Cd G_u^2}{ N \mu^2 t^2}\\
=&
\left(1 - \frac{2}{t}\right)\EE\Big[ \EE\left[f(\x_t) - f(\x^*) \mid \cF_{t-1} \right] \Big] 
+
\frac{2 \EE\left[ \EE\left[\nabla^\top f(\x_t) \EE\left[\tbZ_{\xi_t}^\perp \mid \cF_t\right] \mid \cF_{t-1} \right]  \right]}{N\mu t} \sum_{k\in\cK}\bb_{(k)}^t
+ 
\frac{  L \Cd G_u^2}{ N \mu^2 t^2}\\
&\vdots\\
\leq&
\frac{2}{N\mu}\sum_{i=2}^{t} 
\prod_{j=i+1}^t \left(1 - \frac{2}{j}\right) \frac{\EE\left[ \EE \left[ \nabla^\top f(\x_i) \EE\left[ \bZ_{\xi_i}^\perp \mid \cF_i \right] \mid \cF_{i-1}\right]   \right]}{i} \sum_{k\in\cK}\bb_{(k)}^i
\\
&
+\frac{  L \Cd G_u^2}{4 N \mu^2 } \sum_{i=2}^{t} \frac{1}{i^2} \prod_{j=i+1}^{t} \left(1 - \frac{2}{j}\right),
\end{align*}
where the last inequality is because of Lemma~\ref{lem:recs} with $\Delta_t = \EE\Big[ \EE\left[f(\x_t) - f(\x^*) \mid \cF_{t-1} \right] ]\Big] $, $\rho_t = 1 -\frac{2}{t}$, and $\beta_t = \frac{2 \EE\left[ \EE\left[\nabla^\top f(\x_t) \EE\left[\tbZ_{\xi_t}^\perp \mid \cF_t\right] \mid \cF_{t-1} \right]  \right]}{N\mu t} \sum_{k\in\cK}\bb_{(k)}^t
+ 
\frac{  L \Cd G_u^2}{ N \mu^2 t^2}$.

Furthermore,
\begin{equation*}
	\prod_{j=i+1}^{t} \left(1 - \frac{2}{j}\right) = \prod_{j=i+1}^{t} \frac{j-2}{j} = \frac{(i-1)i}{t(t-1)}.
\end{equation*}
Thus,
\begin{align*}
	\sum_{i=2}^{t} \frac{1}{i^2} \prod_{j=i+1}^{t} \left(1 - \frac{2}{j}\right)
	= 
	\sum_{i=2}^{t} \frac{1}{i^2}\frac{(i-1)i}{t(t-1)}
	\leq 
	\frac{1}{t}.
\end{align*}

Combining above results, we can obtain that
\begin{align*}
	\EE\left[f(\x_{t+1}) - f(\x^*)\right]
	\leq
	\frac{2}{N \mu t(t-1) }\sum_{i=2}^{t} (i-1) \EE\left[ \EE \left[ \nabla^\top f(\x_i) \EE\left[ \bZ_{\xi_i}^\perp \mid \cF_i \right] \mid \cF_{i-1} \right]   \right] \sum_{k\in\cK}\bb_{(k)}^i
	+
	\frac{  L \Cd G_u^2}{4 N \mu^2 t}.
\end{align*}


\end{proof}

\begin{lemma}\label{lem:dec4}
Let  functions $f(\x)$ and $f(\x;\;\xi)$  satisfy the properties in Lemma~\ref{lem:dec1}.
All conditions in Lemma~\ref{lem:dec3} also hold. 
Let $\cE_t = \cap_{i=1,\dots,^5} \cE_{t,i}$ be the event that events $\cE_{t,1}$-$\cE_{t,5}$ all hold. 
We assume that events $\{\cE_{i}\}$ with $i = 1,\dots, t$ all hold.
By setting $\eta_t = \heta_t \norm{\nabla f(\x_t;\;\xi_t)} $ with $\heta_t = \heta$, then it holds that
\begin{equation}\label{eq:dec4}
\begin{aligned}
&\frac{1}{T}\sum_{t=1}^{T} \EE\left[\norm{\nabla f(\x_t)}^2\right]\\
\leq& 
\frac{f(\x_1) - \EE\left[f(\x_{T+1})\right]}{2T\heta}
+ 
\sum_{t=1}^{T}\frac{2  \EE\left[ \EE\left[ \nabla^\top f(\x_t)\EE\left[ \tbZ_{\xi_t}^\perp \mid \cF_t \right] \mid \cF_{t-1} \right]\right]}{NT} \sum_{k\in\cK}\bb_{(k)}^t\\
&
+
\frac{2 \heta L \Cd G_u^2}{N}.
\end{aligned}
\end{equation}
\end{lemma}
\begin{proof}
We can represent Eq.~\eqref{eq:dec3} with $\heta_t = \heta$ as follows
\begin{align*}
	\norm{\nabla f(\x_t)}^2
	\leq 
	\frac{f(\x_t) - \EE\left[ f(\x_{t+1}) \mid \cF_t\right]}{2\heta}
	+ 
	\frac{2 \nabla^\top  f(\x_t) \EE\left[\tbZ_{\xi_t}^\perp \mid \cF_t\right]}{N} \sum_{k\in\cK}\bb_{(k)}^t
	+
	\frac{2 \heta L \Cd G_u^2}{N}.
\end{align*}
By the law of total expectation, we can obtain that
\begin{align*}
	\EE\left[ \norm{\nabla f(\x_t)}^2 \right]
	=&
	\EE\left[\EE\left[ \norm{\nabla f(\x_t)}^2 \mid \cF_{t-1} \right]\right]\\
	\leq&
	\frac{\EE\left[ \EE\left[ f(\x_t) \mid \cF_{t-1} \right] \right] - \EE\left[ \EE\left[ \EE\left[f(\x_{t+1}) \mid \cF_t\right] \mid \cF_{t-1}\right] \right]}{2\heta}\\
	&+ 
	\frac{2 \EE\left[ \EE\left[ \nabla^\top f(\x_t)\EE\left[ \tbZ_{\xi_t}^\perp \mid \cF_t \right] \mid \cF_{t-1} \right]\right] }{N}\sum_{k\in\cK}\bb_{(k)}^t
	+
	\frac{2 \heta L \Cd G_u^2}{N}\\
	=&
	\frac{\EE\left[ f(\x_t) \right] - \EE\left[ f(\x_{t+1}) \right]}{2\heta} 
	+ 
	\frac{2 \EE\left[ \EE\left[ \nabla^\top f(\x_t)\EE\left[ \tbZ_{\xi_t}^\perp \mid \cF_t \right] \mid \cF_{t-1} \right]\right] }{N}\sum_{k\in\cK}\bb_{(k)}^t\\
	&+
	\frac{2 \heta L \Cd G_u^2}{N}.
\end{align*}

Telescoping above equation, we can obtain that
\begin{align*}
	&\frac{1}{T}\sum_{t=1}^{T} \EE\left[\norm{\nabla f(\x_t)}^2\right]\\
	\leq& 
	\frac{f(\x_1) - \EE\left[f(\x_{T+1})\right]}{2T\heta}
	+ 
	\sum_{t=1}^{T}\frac{2  \EE\left[ \EE\left[ \nabla^\top f(\x_t)\EE\left[ \tbZ_{\xi_t}^\perp \mid \cF_t \right] \mid \cF_{t-1} \right]\right]}{NT} \sum_{k\in\cK}\bb_{(k)}^t\\
	&
	+
	\frac{2 \heta L \Cd G_u^2}{N}.
\end{align*}
\end{proof}

Finally, we will define two new events that will also hold with a high probability.

\begin{lemma}\label{lem:gauss_bnd}
Let the objective function $f(\x)$ be $L$-smooth and Assumption~\ref{ass:bsm} hold.	Given $0<\delta<1$, with a probability at least $1-\delta$, it holds that
\begin{equation}\label{eq:gauss_bnd}
	\begin{aligned}
		&\left|	\sum_{i=2}^{t} (i-1)  \EE\left[ \EE \left[ \nabla^\top f(\x_i) \EE\left[ \bZ_{\xi_i}^\perp \mid \cF_i \right] \mid \cF_{i-1} \right]   \right]\sum_{k\in\cK}\bb_{(k)}^i \right|\\
		\leq& 
		\sqrt{ 8L N G_u^2 \cdot \sum_{i=2}^{t} (i-1)^2 \cdot \EE\Big[f(\x_i) - f(\x^*)\Big] \cdot\log\frac{2}{\delta} }.
	\end{aligned}
\end{equation}
Let $\cE_{t,6}$ be the event that Eq.~\eqref{eq:gauss_bnd} holds. Then we have $\Pr(\cE_{t,6}) \geq 1 - \delta$.
\end{lemma}
\begin{proof}
	First, by the definition of $\bb_{(k)}^t = [\bZ_t^\perp]^\top \ub_{(k)}$ in Lemma~\ref{lem:d_nab}, we can obtain that
\begin{align*}
	\sum_{k\in\cK}\bb_{(k)}^i \sim \cN(\bm{0}, \frac{N}{2} \bm{I}_{d-1}).
\end{align*}
Accordingly, we can obtain that
\begin{align*}
	\EE\left[ \EE \left[ \nabla^\top f(\x_i) \EE\left[ \bZ_{\xi_i}^\perp \mid \cF_i \right] \mid \cF_{i-1} \right]   \right] \sum_{k\in\cK}\bb_{(k)}^i \sim \cN\left(0, \frac{N}{2}\norm{\EE\left[ \EE \left[ \nabla^\top f(\x_i) \EE\left[ \bZ_{\xi_i}^\perp \mid \cF_i \right] \mid \cF_{i-1} \right]   \right]}^2\right)
\end{align*}
Furthermore, for different iteration $i$, $\sum_{k\in\cK}\bb_{(k)}^i$ is independent to each other. 
Thus, we can obtain that
\begin{align*}
	&\sum_{i=2}^{t} (i-1)  \EE\left[ \EE \left[ \nabla^\top f(\x_i) \EE\left[ \bZ_{\xi_i}^\perp \mid \cF_i \right] \mid \cF_{i-1} \right]   \right]\sum_{k\in\cK}\bb_{(k)}^i\\
	&\sim 
	\cN\left(0, \frac{N}{2}\sum_{i=2}^{t} (i-1)^2  \norm{\EE\left[ \EE \left[ \nabla^\top f(\x_i) \EE\left[ \bZ_{\xi_i}^\perp \mid \cF_i \right] \mid \cF_{i-1} \right]   \right]}^2
	\right).
\end{align*}
By Lemma~\ref{lem:gaussian-tail}, we can obtain that  with a probability at least $1-\delta$ that
\begin{align*}
	&\left|	\sum_{i=2}^{t} (i-1)  \EE\left[ \EE \left[ \nabla^\top f(\x_i) \EE\left[ \bZ_{\xi_i}^\perp \mid \cF_i \right] \mid \cF_{i-1} \right]   \right]\sum_{k\in\cK}\bb_{(k)}^i \right|\\
	\leq&
	\sqrt{ \left( \frac{N}{2}\sum_{i=2}^{t} (i-1)^2 \norm{ \EE\left[ \EE \left[ \nabla^\top f(\x_i) \EE\left[ \bZ_{\xi_i}^\perp \mid \cF_i \right] \mid \cF_{i-1} \right]   \right]}^2 \right) \cdot 2 \log\frac{2}{\delta}} \\
	\leq&
	\sqrt{   \left( N\sum_{i=2}^{t} (i-1)^2 \EE\left[ \EE\left[\EE\left[\norm{ \nabla^\top f(\x_i) \tbZ_{\xi_i}^\perp}^2  \mid \cF_i \right] \mid \cF_{i-1} \right]    \right]\right)\cdot \log\frac{2}{\delta}},
\end{align*}
where the last inequality is because of Jensen's inequality.

Furthermore,
\begin{equation}\label{eq:ex}
	\begin{aligned}
		&\EE\left[\norm{ \nabla^\top f(\x_i) \tbZ_{\xi_i}^\perp}^2 \mid \cF_i \right]\\
		\stackrel{\eqref{eq:tbz}}{=}&
		\EE 
		\left[\norm{ \norm{\nabla f(\x_i;\;\xi_i)} \nabla f(\x_i) - \frac{\dotprod{\nabla f(\x_i;\;\xi_i), \nabla f(\x_i)}}{\norm{\nabla f(\x_i;\;\xi_i)}} \nabla f(\x_i;\;\xi_i) }^2 \mid \cF_i\right]\\
		\leq&
		2 \EE\left[ \norm{\nabla f(\x_i;\;\xi_i)}^2 \norm{\nabla f(\x_i)}^2 + \frac{\norm{\nabla f(\x_i;\;\xi_i)}^4 \norm{\nabla f(\x_i)}^2}{\norm{\nabla f(\x_i;\;\xi_i)}^2} \mid \cF_i \right] \\
		\leq&
		4 G_u^2 \EE\left[ \norm{\nabla f(\x_i)}^2 \right],
	\end{aligned}
\end{equation}
where the last inequality is because of Assumption~\ref{ass:bsm}. 

Accordingly, we have
\begin{align*}
	&\EE\left[ \EE\left[\EE\left[\norm{ \nabla^\top f(\x_i) \tbZ_{\xi_i}^\perp}^2  \mid \cF_i \right] \mid \cF_{i-1} \right]  \right]
	\stackrel{\eqref{eq:ex}}{\leq}
	4 G_u^2 \EE\left[ \EE\left[ \norm{\nabla f(\x_i)}^2 \mid \cF_{i-1} \right]\right]\\
	\leq&
	8L G_u^2 \EE\Big[ \EE\left[ f(\x_i) - f(\x^*) \right] \mid \cF_{i-1} \Big]
	=
	8L G_u^2 \EE\Big[ f(\x_i) - f(\x^*) \Big],
\end{align*}
where the last inequality is because of Lemma~\ref{lem:L_smth}.

Therefore, we can obtain that with a probability at least $1-\delta$, it holds that
\begin{align*}
	&\left|	\sum_{i=2}^{t} (i-1)  \EE\left[ \EE \left[ \nabla^\top f(\x_i) \EE\left[ \bZ_{\xi_i}^\perp \mid \cF_i \right] \mid \cF_{i-1} \right]   \right]\sum_{k\in\cK}\bb_{(k)}^i \right|\\
	\leq& 
	\sqrt{ 8L N G_u^2 \cdot \sum_{i=2}^{t} (i-1)^2 \cdot \EE\Big[f(\x_i) - f(\x^*)\Big] \cdot\log\frac{2}{\delta} }.
\end{align*}

\end{proof}

\begin{lemma}
Let Assumption~\ref{ass:bsm} hold.	Given $0<\delta<1$, with a probability at least $1-\delta$, it holds that
\begin{equation}\label{eq:gauss_bndn}
\left|
\sum_{t=1}^{T}  \EE\left[ \EE \left[ \nabla^\top f(\x_t) \EE\left[ \bZ_{\xi_t}^\perp \mid \cF_t \right] \mid \cF_{t-1} \right]   \right]\sum_{k\in\cK}\bb_{(k)}^t
\right|
\leq
\sqrt{ 4 NG_u^2 \cdot \sum_{t=1}^{T} \EE\left[\norm{\nabla f(\x_t)}^2\right] \cdot \log\frac{2}{\delta}}.
\end{equation}
Let $\cE_{t,7}$ be the event that Eq.~\eqref{eq:gauss_bndn} holds. Then we have $\Pr(\cE_{t,7}) \geq 1 - \delta$.
\end{lemma}
\begin{proof}
The proof is similar to that of Lemma~\ref{lem:gauss_bnd}.
We have
\begin{align*}
	\sum_{t=1}^{T}  
	\EE\left[ \EE \left[ \nabla^\top f(\x_t) \EE\left[ \bZ_{\xi_t}^\perp \mid \cF_t \right] \mid \cF_{t-1} \right] \right] 
	\sum_{k\in\cK}\bb_{(k)}^t
	\sim 
	\cN\left(0, \frac{N}{2}\sum_{t=1}^{T}  \norm{\EE\left[ \EE \left[ \nabla^\top f(\x_t) \EE\left[ \bZ_{\xi_t}^\perp \mid \cF_t \right] \mid \cF_{t-1} \right]   \right]}^2
	\right).
\end{align*}
By Lemma~\ref{lem:gaussian-tail}, we can obtain that  with a probability at least $1-\delta$ that
\begin{align*}
&\left|
\sum_{t=1}^{T}  \EE\left[ \EE \left[ \nabla^\top f(\x_t) \EE\left[ \bZ_{\xi_t}^\perp \mid \cF_t \right] \mid \cF_{t-1} \right]   \right]\sum_{k\in\cK}\bb_{(k)}^t
\right|\\
\leq&
\sqrt{2\log\frac{2}{\delta} \cdot  \frac{N}{2}\sum_{t=1}^{T}  \norm{\EE\left[ \EE \left[ \nabla^\top f(\x_t) \EE\left[ \bZ_{\xi_t}^\perp \mid \cF_t \right] \mid \cF_{t-1} \right]   \right]}^2 }\\
\stackrel{\eqref{eq:ex}}{\leq}&
\sqrt{ 4 NG_u^2   \log\frac{2}{\delta} \cdot \sum_{t=1}^{T} \EE\left[\norm{\nabla f(\x_t)}^2\right]}.
\end{align*}
\end{proof}

\subsection{Main Theorems and Query Complexities}

First, we will give two main theorems which describe the convergence properties of Algorithm~\ref{alg:SA} when the objective function is $L$-smooth and $\mu$-strongly convex, and $L$-smooth but maybe nonconvex, respectively.

\begin{theorem}\label{thm:main}
	Assume that $f(\x)$ is $L$-smooth and $\mu$-strongly convex. 
	Each $f(\x;\;\xi)$ is also $L$-smooth.
	Let the stochastic gradient $\nabla f(\x;\;\xi)$ be of bounded second moment and its norm is lower bounded by $G_\ell$, that is, Assumption~\ref{ass:bsm} and Assumption~\ref{ass:lb} hold. 
	Given $0<\delta<1$, let $\cE_t = \cap_{i=1,\dots,6} \cE_{t,i}$ be the event that events $\cE_{t,1}$-$\cE_{t,6}$ (defined in Eq.~\eqref{eq:E1} - Eq.~\eqref{eq:E5} and Lemma~\ref{lem:gauss_bnd}, respectively) all hold. 
	Let events $\cE_t$ hold for $t = 1,\dots, T-1$. 
	Set the smooth parameter $\alpha$ satisfying $\alpha \leq \frac{G_\ell}{2\Cd L}$ with $\Cd$ defined in Eq.~\eqref{eq:d_up} and step size $\eta_t = \frac{\norm{\nabla f(\x_t;\;\xi_t)}}{2\mu t}$.
	Then Algorithm~\ref{alg:SA1} has the following convergence property
\begin{equation}\label{eq:main}
\EE\Big[f(\x_T) - f(\x^*)\Big]
\le 
\frac{\max \{\Us,\; f(\x_1) - f(\x^*)\}}{T},
\end{equation}
where we take the expectation  with respect to $\xi$ and $U_s$ is defined as 
\begin{equation}
\Us:= \frac{  3dL  G_u^2}{4 N \mu^2 } 
+ \frac{  3L G_u^2 \log\frac{1}{\delta}}{2 N \mu^2 }
+
\frac{72 L G_u^2 \log\frac{2}{\delta}}{N \mu^2}. 
\end{equation}
Furthermore,
\begin{equation}\label{eq:P}
\Pr\left(\cap_{t=1}^{T-1} \cE_t\right) \geq 1 - T \left( \frac{(N+2)\delta}{2} +  2 \exp\left( -N \cdot K(1/4 \Vert p) \right)\right), \mbox{ with } p = 0.0224
\end{equation}
where the binary Kullback--Leibler divergence is defined as
\begin{equation}\label{eq:KL1}
	K(q \Vert p) \;=\; q\ln\frac{q}{p}+(1-q)\ln\frac{1-q}{1-p}.
\end{equation}
\end{theorem}
\begin{proof}
First, all conditions required in Lemma~\ref{lem:dec_2} are satisfied. 
Thus, we have
\begin{align*}
&\EE\left[f(\x_{t+1}) - f(\x^*)\right]\\
\stackrel{\eqref{eq:dec2}}{\leq}&
\frac{2}{N \mu t(t-1) }\sum_{i=2}^{t} (i-1) \EE\left[ \EE \left[ \nabla^\top f(\x_i) \EE\left[ \bZ_{\xi_i}^\perp \mid \cF_i \right] \mid \cF_{i-1} \right]   \right] \sum_{k\in\cK}\bb_{(k)}^i
+
\frac{  L \Cd G_u^2}{4 N \mu^2 t}\\
\stackrel{\eqref{eq:gauss_bnd}}{\leq}&
\frac{2}{N \mu t(t-1) }\sqrt{ 8L N G_u^2 \cdot \sum_{i=2}^{t} (i-1)^2 \EE\Big[f(\x_i) - f(\x^*)\Big] \cdot\log\frac{2}{\delta} }
+
\frac{  L \Cd G_u^2}{4 N \mu^2 t},
\end{align*}
where the last inequality is because of Lemma~\ref{lem:gauss_bnd} and event $\cE_{t,6}$.

With $b$ and $c$ being set as 
\begin{align*}
	b = 
	\frac{4G_u}{\sqrt{N}\mu} \sqrt{ 2 L \log\frac{2}{\delta}},\quad
	c = \frac{  L \Cd G^2}{4 N \mu^2 }, \quad\mbox{ and }\quad \Delta_t = \EE\Big[f(\x_t) - f(\x^*)\Big],
\end{align*}
 then by Lemma~\ref{lem:dec2}, we can obtain that for any $t\geq 2$, it holds that
\begin{align*}
\EE\left[f(\x_t) - f(\x^*)\right]
\leq& 
\frac{1}{t} \left( \frac{9b^2}{4} + 3c \right)
= 
\frac{1}{t}\left(
\frac{72 L G_u^2 \log\frac{2}{\delta}}{N \mu^2}
+
\frac{  3L \Cd G_u^2}{4 N \mu^2 }
\right)\\
=&
\frac{1}{t}\left(\frac{  3dL  G_u^2}{4 N \mu^2 } 
+ \frac{  3L G_u^2 \log\frac{1}{\delta}}{2 N \mu^2 }
+
\frac{72 L G_u^2 \log\frac{2}{\delta}}{N \mu^2}\right),
\end{align*}
where the last equality is because of the definition of $\Cd$ in Eq.~\eqref{eq:d_up}.
Above result holds for $t\ge 2$. 
For $t = 1$, we have
\begin{align*}
	\EE\left[f(\x_1) - f(\x^*)\right]
	\leq 
	\frac{f(\x_1) - f(\x^*)}{1}.
\end{align*}
Combining above results, we can the result in Eq.~\eqref{eq:main}.

Next, we will bound $\Pr(\cE_t)$ where $\cE_t$ is the event that  events $\cE_{t,1}$-$\cE_{t,6}$ all hold.
By a union bound, we have
\begin{align*}
\Pr(\cE_t) 
=& 
\Pr\left(\forall_i \mbox{ event } \cE_{t,i}\mbox{ holds}\right)\\
\geq&
1 - \left(\frac{N\delta}{2} + 2 \exp\left( -N \cdot K(1/4 \Vert p) \right) + \delta\right)\\
=&
1 - \left( \frac{(N+2)\delta}{2} +  2 \exp\left( -N \cdot K(1/4 \Vert p) \right)\right),
\end{align*}
where the inequality is because of probabilities of events $\cE_{t,1}$-$\cE_{t,6}$ shown in Lemma~\ref{lem:E1}-Lemma~\ref{lem:E5}, and Lemma~\ref{lem:gauss_bnd}.

Using the union bound again, we can obtain that
\begin{align*}
\Pr\left(\cap_{t=1}^T \cE_t\right) \geq 1 - T \left( \frac{(N+2)\delta}{2} +  2 \exp\left( -N \cdot K(1/4 \Vert p) \right)\right), \mbox{ with } p = 0.0224,
\end{align*}
which concludes the proof.
\end{proof}

\begin{theorem}\label{thm:main1}
	Assume that $f(\x)$ and each $f(\x;\;\xi)$ are $L$-smooth. 
	Assume that $f(\x)$ is low bounded, that is, $f(\x^*) \geq -\infty$.
	Let the stochastic gradient $\nabla f(\x;\;\xi)$ be of bounded second moment and its norm is lower bounded by $G_\ell$, that is, Assumption~\ref{ass:bsm} and Assumption~\ref{ass:lb} hold.. 
	Given $0<\delta<1$, let $\cE_t = \cap_{i=1,\dots,5,7} \cE_{t,i}$ be the event that events $\cE_{t,1}$-$\cE_{t,5}$ (defined in Eq.~\eqref{eq:E1} - Eq.~\eqref{eq:E5}) and $\cE_{t,7}$ (defined in Lemma~\ref{lem:gauss_bnd}) all hold. 
	Let events $\cE_t$ hold for $t = 1,\dots, T-1$. 
	Set the smooth parameter $\alpha$ satisfying $\alpha \leq \frac{G_\ell}{2\Cd L}$ and step size $\eta_t = \frac{\sqrt{N}\norm{\nabla f(\x_t;\;\xi_t)}}{\sqrt{T}} \cdot \frac{\sqrt{f(\x_1) - f(\x^*)}}{G_u}$.
	Then Algorithm~\ref{alg:SA1} has the following convergence property
\begin{equation}\label{eq:main1}
\frac{1}{T}\sum_{t=1}^{T} \EE\left[\norm{\nabla f(\x_t)}^2\right]
\leq
\frac{9\sqrt{(f(\x_1) - f(\x^*)) L\Cd} G_u}{\sqrt{T N}}
+ 
\frac{16G_u^2 \log\frac{2}{\delta}}{N T},
\end{equation}
where $\Cd$ is defined in Eq.~\eqref{eq:d_up}.
Furthermore,
\begin{equation}\label{eq:P1}
	\Pr\left(\cap_{t=1}^{T-1} \cE_t\right) \geq 1 - T \left( \frac{(N+2)\delta}{2} +  2 \exp\left( -N \cdot K(1/4 \Vert p) \right)\right), \mbox{ with } p = 0.0224.
\end{equation}
\end{theorem}
\begin{proof}
First, all conditions required in Lemma~\ref{lem:dec4} are satisfied. 
Thus,  we can obtain that
\begin{align*}
	&\frac{1}{T}\sum_{t=1}^{T} \EE\left[\norm{\nabla f(\x_t)}^2\right]\\
	\stackrel{\eqref{eq:dec4}}{\leq}& 
	\frac{f(\x_1) - \EE\left[f(\x_{T+1})\right]}{2T\heta}
	+ 
	\sum_{t=1}^{T}\frac{2  \EE\left[ \EE\left[ \nabla^\top f(\x_t)\EE\left[ \tbZ_{\xi_t}^\perp \mid \cF_t \right] \mid \cF_{t-1} \right]\right]}{NT} \sum_{k\in\cK}\bb_{(k)}^t\\
	&
	+
	\frac{2 \heta L \Cd G_u^2}{N}\\
	\stackrel{\eqref{eq:gauss_bndn}}{\leq}&
	\frac{f(\x_1) - \EE\left[f(\x_{T+1})\right]}{2T\heta}
	+ 
	\frac{2  \sqrt{ 4 NG_u^2   \log\frac{2}{\delta} \cdot \sum_{t=1}^{T} \EE\left[\norm{\nabla f(\x_t)}^2\right]}}{NT} 
	+
	\frac{2 \heta L \Cd G_u^2}{N}\\
	\leq&
	\frac{f(\x_1) - \EE\left[f(\x_{T+1})\right]}{2T\heta}
	+ 
	\frac{\sum_{t=1}^{T} \EE\left[\norm{\nabla f(\x_t)}^2\right]}{2T}
	+ 
	\frac{8G_u^2 \log\frac{2}{\delta}}{N T} 
	+
	\frac{2 \heta L \Cd G_u^2}{N},
\end{align*}
where the last inequality is because of the fact $\sqrt{ab} \leq \frac{a+b}{2}$ holds for all $a, b \geq 0$.

Thus, we can obtain that
\begin{align*}
	\frac{1}{T}\sum_{t=1}^{T} \EE\left[\norm{\nabla f(\x_t)}^2\right]
	\leq& 
	\frac{f(\x_1) - \EE\left[f(\x_{T+1})\right]}{T\heta}
	+ 
	\frac{16G_u^2 \log\frac{2}{\delta}}{N T} 
	+
	\frac{8 \heta L \Cd G_u^2}{N}\\
	\leq&
	\frac{f(\x_1) - f(\x^*)}{T\heta}
	+ 
	\frac{16G_u^2 \log\frac{2}{\delta}}{N T} 
	+
	\frac{8 \heta L \Cd G_u^2}{N}\\
	=&
	\frac{9\sqrt{(f(\x_1) - f(\x^*)) L\Cd} G_u}{\sqrt{T N}}
	+ 
	\frac{16G_u^2 \log\frac{2}{\delta}}{N T},
\end{align*}
where the second inequality is because of $\EE\left[f(\x_{T+1})\right] \geq f(\x^*)$ and last equality is because of $\heta = \sqrt{\frac{N}{T L \Cd}} \cdot \frac{\sqrt{f(\x_1) - f(\x^*)}}{G_u}$.
\end{proof}

\begin{remark}
Our convergence analysis requires Assumption~\ref{ass:lb}, that is, $\norm{\nabla f(\x;\;,\xi)}\geq G_\ell$. 
This assumption is not common in the convergence analysis of stochastic gradient descent.
This assumption is only used to choose the smooth parameter $0<\alpha \leq\frac{G_\ell}{2\Cd L}$ of Algorithm~\ref{alg:SA1}.
Thanks to the well-chosen $\alpha$, our convergence rates do \emph{not} explicitly depend on $\alpha$.  

Though Assumption~\ref{ass:lb} is not common in the convergence analysis of stochastic gradient descent, this assumption  is also reasonable because the stochastic gradient  $\nabla f(\x;\;\xi)$ never vanishes even $\x_t$ tends to the optimal point \citep{bottou2018optimization,moulines2011non}.
\end{remark}

Theorem~\ref{thm:main}  provides the convergence rate of Algorithm~\ref{alg:SA1} when $f(\x)$ is both $L$-smooth and $\mu$-strongly convex.
Theorem~\ref{thm:main1} provides the convergence rate when $f(\x)$ is only $L$-smooth.
In the following two corollaries, we will provide the iteration complexities and query complexities of Algorithm~\ref{alg:SA1}.

\begin{corollary}\label{cor:main}
Assume that $f(\x)$ is $L$-smooth and $\mu$-strongly convex. 
Each $f(\x;\;\xi)$ is also $L$-smooth. 
All conditions in Theorem~\ref{thm:main} are all satisfied and parameters of Algorithm~\ref{alg:SA1} are set as Theorem~\ref{thm:main}.
Set sample size $N = (K(1/4 \Vert p))^{-1} \log (40T)$ with $p = 0.0224$ and the probability parameter $\delta = \frac{1}{20 NT}$. 
Furthermore, assume that the dimension is sufficiently large that $d \geq \log(40 N T) $.   
Then, with a probability at least $0.8$, Algorithm~\ref{alg:SA1} can find a $\x_T$ satisfies
\begin{equation}
	f(\x_T) - f(\x^*) \leq \varepsilon,
\end{equation}
with the iteration number $T$ satisfies
\begin{equation}\label{eq:T}
	T 
	=   
	\frac{10LG_u^2}{N\mu^2 \varepsilon} 
	\cdot 
	\left(\frac{  3d  }{4   } 
	+ 
	\frac{  3  \log\frac{1}{\delta}}{2   }
	+
	72\log\frac{2}{\delta}\right)
	+
	\frac{10(f(\x_1) - f(\x^*))}{\varepsilon}.
\end{equation} 
The total queries are
\begin{equation}\label{eq:Q1}
	Q = \frac{750 d LG_u^2}{\mu^2 \varepsilon}
	+ \frac{10(f(\x_1) - f(\x^*)) \log\left( \max\{\frac{75 d LG_u^2}{\mu^2 \varepsilon}, f(\x_1) - f(\x^*)  \} \right)}{K(1/4 \Vert p) \cdot \varepsilon}.
\end{equation}
\end{corollary}
\begin{proof}
Conditioned on events $\cE_t$ hold for $t = 1,\dots, T-1$, Eq.~\eqref{eq:main} holds.
By Markov's inequality, we can obtain that
\begin{equation}\label{eq:PP}
\Pr\left( f(\x_T) - f(\x^*) > \varepsilon \mid  \cap_{t=1}^{T-1} \cE_t\right) 
\leq
\frac{\EE\left[ f(\x_T) - f(\x^*) \right]}{\varepsilon} 
\stackrel{\eqref{eq:main}}{\leq}
\frac{\max\{U_s, f(\x_1) - f(\x^*)\}}{T\varepsilon}
=
\frac{1}{10}.
\end{equation}
Therefore, we only need 
\begin{align}
	T 
	=& 
	\frac{10 \max\{U_s, f(\x_1) - f(\x^*)\}}{\varepsilon} \notag \\
	\leq&
	\frac{10LG_u^2}{N\mu^2 \varepsilon} 
	\cdot 
	\left(\frac{  3d  }{4   } 
	+ 
	\frac{  3  \log\frac{1}{\delta}}{2   }
	+
	72\log\frac{2}{\delta}\right)
	+
	\frac{10(f(\x_1) - f(\x^*))}{\varepsilon}. \label{eq:T1}
\end{align}

At the same time, we require that $\Pr\left(\cap_{t=1}^{T-1} \cE_t\right) \geq 1 - \frac{1}{10}$.
Eq.~\eqref{eq:P} implies that we require that
\begin{align*}
T \left( \frac{(N+2)\delta}{2} +  2 \exp\left( -N \cdot K(1/4 \Vert p) \right)\right) 
\leq& 
\frac{1}{10}.
\end{align*}
Replacing $\delta = \frac{1}{20NT}$ and  $N = (K(1/4 \Vert p))^{-1} \log (40T)$ to the left-hand side of the above equation, we can obtain that
\begin{align*}
	T \left( \frac{(N+2)\delta}{2} +  2 \exp\left( -N \cdot K(1/4 \Vert p) \right)\right)
	\leq \frac{1}{20}
	+ 2 T \cdot \exp\left( \log - (40 T) \right)
	=
	\frac{1}{10}.
\end{align*}
Thus, it holds that $\Pr\left(\cap_{t=1}^{T-1} \cE_t\right) \geq 1 - \frac{1}{10}$.
Combining with Eq.~\eqref{eq:PP} and the law of total probability, we have
\begin{align*}
\Pr\left( f(\x_T) - f(\x^*) \leq \varepsilon \right)
=& 
\Pr\left( f(\x_T) - f(\x^*) \leq \varepsilon \mid  \cap_{t=1}^{T-1} \cE_t \right) 
\cdot 
\Pr\left(\cap_{t=1}^{T-1} \cE_t\right)\\
&+
\Pr\left( f(\x_T) - f(\x^*) \leq \varepsilon \mid  \overline{\cap_{t=1}^{T-1} \cE_t} \right) 
\cdot 
\Pr\left(\overline{\cap_{t=1}^{T-1} \cE_t}\right)\\
\geq&
\Pr\left( f(\x_T) - f(\x^*) \leq \varepsilon \mid  \cap_{t=1}^{T-1} \cE_t \right) 
\cdot 
\Pr\left(\cap_{t=1}^{T-1} \cE_t\right)\\
=& \left(1 - \frac{1}{10}\right)^2
\geq 0.8,
\end{align*}
where $\overline{\cap_{t=1}^{T-1} \cE_t}$ is the complement of $\cap_{t=1}^{T-1} \cE_t$.

The total queries are 
\begin{align*}
	Q 
	=& T N 
	\stackrel{\eqref{eq:T1}}{\leq} 
	\frac{10LG_u^2}{\mu^2 \varepsilon} 
	\cdot 
	\left(\frac{  3d  }{4   } 
	+ 
	\frac{  3  \log\frac{1}{\delta}}{2   }
	+
	72\log\frac{2}{\delta}\right)
	+
	\frac{10N(f(\x_1) - f(\x^*))}{\varepsilon}\\
	\leq&
	\frac{750 d LG_u^2}{\mu^2 \varepsilon}
	+ \frac{10(f(\x_1) - f(\x^*)) \log\left( \max\{\frac{75 d LG_u^2}{\mu^2 \varepsilon}, f(\x_1) - f(\x^*)  \} \right)}{K(1/4 \Vert p) \cdot \varepsilon},
\end{align*}
where the last inequality is because of the assumption that $d \geq \log(40 N T) $ and $\delta = \frac{1}{20NT}$.

\end{proof}

\begin{corollary}\label{cor:main1}
Assume that $f(\x)$ and each $f(\x;\;\xi)$ are $L$-smooth. 
All conditions in Theorem~\ref{thm:main1} are all satisfied and parameters of Algorithm~\ref{alg:SA1} are set as Theorem~\ref{thm:main1}.
Set sample size $N = (K(1/4 \Vert p))^{-1} \log (40T)$ with $p = 0.0224$ and the probability parameter $\delta = \frac{1}{20 NT}$. 
Furthermore, assume that the dimension is sufficiently large that $d \geq \log(40 N T) $.   
Then, with a probability at least $0.8$, the sequence $\{\x_t\}$ with $t=1,\dots, T$ generated by Algorithm~\ref{alg:SA1} satisfies
\begin{equation}
\frac{1}{T}\sum_{t=1}^{T}\norm{\nabla f(\x_t)}^2  \leq \varepsilon,
\end{equation}
with $0<\varepsilon<1$, if the iteration number $T$ satisfies
\begin{equation}\label{eq:T2}
T = \frac{ 180^2 (f(\x_1) - f(\x^*)) d L  G_u^2}{N \varepsilon^2} 
+
\frac{2\cdot 180^2 (f(\x_1) - f(\x^*)) L  G_u^2\log\frac{1}{\delta}}{N \varepsilon^2} 
+ 
\frac{320 G_u^2 \log\frac{2}{\delta}}{N\varepsilon}.
\end{equation}
Furthermore, the total queries are
\begin{equation}\label{eq:Q2}
Q = \frac{ 3\cdot 180^2 (f(\x_1) - f(\x^*)) d L  G_u^2}{ \varepsilon^2} 
+
\frac{320 d G_u^2 }{\varepsilon}.
\end{equation}
\end{corollary}
\begin{proof}
Conditioned on events $\cE_t$ hold for $t = 1,\dots, T-1$, Eq.~\eqref{eq:main1} holds.	
By  Markov's inequality, we can obtain that
\begin{equation}\label{eq:P2}
\Pr\left( \frac{1}{T}\sum_{t=1}^{T}\norm{\nabla f(\x_t)}^2  \geq \varepsilon \mid \cap_{t=1}^{T-1} \cE_t \right)
\leq 
\frac{\frac{1}{T}\sum_{t=1}^{T}\EE\left[\norm{\nabla f(\x_t)}^2\right]}{\varepsilon} = \frac{1}{10}.
\end{equation}
Then, by Eq.~\eqref{eq:main1}, we only need that
\begin{align*}
	\frac{9\sqrt{(f(\x_1) - f(\x^*)) L\Cd} G_u}{\sqrt{T N}}
	+ 
	\frac{16G_u^2 \log\frac{2}{\delta}}{N T} 
	=
	\frac{\varepsilon}{20} + \frac{\varepsilon}{20} = \frac{\varepsilon}{10}.
\end{align*}
Therefore, $T$ only needs 
\begin{align*}
	T 
	\geq& 
	\frac{180^2 (f(\x_1) - f(\x^*)) L\Cd G_u^2}{N \varepsilon^2} 
	+ 
	\frac{320 G_u^2 \log\frac{2}{\delta}}{N\varepsilon}\\
	=&
	\frac{ 180^2 (f(\x_1) - f(\x^*)) d L  G_u^2}{N \varepsilon^2} 
	+
	\frac{2\cdot 180^2 (f(\x_1) - f(\x^*)) L  G_u^2\log\frac{1}{\delta}}{N \varepsilon^2} 
	+ 
	\frac{320 G_u^2 \log\frac{2}{\delta}}{N\varepsilon}. 
\end{align*}

Furthermore, Eq.~\eqref{eq:P1} implies that we require  $\Pr\left(\cap_{t=1}^{T-1} \cE_t\right) \geq 1 - \frac{1}{10}$.
Accordingly,
\begin{align*}
	T \left( \frac{(N+2)\delta}{2} +  2 \exp\left( -N \cdot K(1/4 \Vert p) \right)\right) 
	\leq& 
	\frac{1}{10}.
\end{align*}
Replacing $\delta = \frac{1}{20NT}$ and  $N = (K(1/4 \Vert p))^{-1} \log (40T)$ to the left hand of above equation, we can obtain that
\begin{align*}
	T \left( \frac{(N+2)\delta}{2} +  2 \exp\left( -N \cdot K(1/4 \Vert p) \right)\right)
	\leq \frac{1}{20}
	+ 2 T \cdot \exp\left( \log - (40 T) \right)
	=
	\frac{1}{10}.
\end{align*}
Thus, it holds that $\Pr\left(\cap_{t=1}^{T-1} \cE_t\right) \geq 1 - \frac{1}{10}$.

Combining with Eq.~\eqref{eq:P2} and the law of total probability, we have
\begin{align*}
\Pr\left( \frac{1}{T}\sum_{t=1}^{T}\norm{\nabla f(\x_t)}^2  \leq \varepsilon \right)
=&
\Pr\left( \frac{1}{T}\sum_{t=1}^{T}\norm{\nabla f(\x_t)}^2  \leq \varepsilon \mid \cap_{t=1}^{T-1} \cE_t \right) 
\cdot 
\Pr\left(\cap_{t=1}^{T-1} \cE_t\right)\\
&
+
\Pr\left( \frac{1}{T}\sum_{t=1}^{T}\norm{\nabla f(\x_t)}^2  \leq \varepsilon \mid \overline{\cap_{t=1}^{T-1} \cE_t} \right) 
\cdot 
\Pr\left(\overline{\cap_{t=1}^{T-1} \cE_t}\right)\\
\geq&
\Pr\left( \frac{1}{T}\sum_{t=1}^{T}\norm{\nabla f(\x_t)}^2  \leq \varepsilon \mid \cap_{t=1}^{T-1} \cE_t \right) 
\cdot 
\Pr\left(\cap_{t=1}^{T-1} \cE_t\right)\\
\geq&
\left(1 - \frac{1}{10}\right)^2
\geq 0.8,
\end{align*}
where $\overline{\cap_{t=1}^{T-1} \cE_t}$ is the complement of $\cap_{t=1}^{T-1} \cE_t$.

The total queries are 
\begin{align*}
	Q 
	= TN
	\stackrel{\eqref{eq:T2}}{=}&
	\frac{ 180^2 (f(\x_1) - f(\x^*)) d L  G_u^2}{ \varepsilon^2} 
	+
	\frac{2\cdot 180^2 (f(\x_1) - f(\x^*)) L  G_u^2\log\frac{1}{\delta}}{ \varepsilon^2} 
	+ 
	\frac{320 G_u^2 \log\frac{2}{\delta}}{\varepsilon}  \\
	\leq&
	\frac{ 3\cdot 180^2 (f(\x_1) - f(\x^*)) d L  G_u^2}{ \varepsilon^2} 
	+
	\frac{320 d G_u^2 }{\varepsilon},
\end{align*}
where the last inequality is because of the assumption that $d \geq \log(40 N T) $ and $\delta = \frac{1}{20NT}$.
\end{proof}

\begin{remark}
Both Eq.~\eqref{eq:T} and Eq.~\eqref{eq:T2} show that a large sample size $N$ will reduce the iteration number. 
However, the sample size $N$ will \emph{not} affect the total query complexities just as implied by Eq.~\eqref{eq:Q1} and Eq.~\eqref{eq:Q2}.
\end{remark}

\begin{remark}
If $f(\x)$ is $L$-smoothness and $\mu$-strong convex,  Corollary~\ref{cor:main} shows that Algorithm~\ref{alg:SA1} can achieve a query complexity $\cO\left( \frac{dLG_u^2}{\mu^2\varepsilon}\right)$ to find a $\x_T$ satisfying $f(\x_T) - f(\x^*) \leq \varepsilon$. 
This query complexity is the same as that of the value-based ZO algorithm \citep{wangadvancement}.

If the objective function may be nonconvex, to find an $\varepsilon$-stationary point, that is $\frac{1}{T}\sum_{t=1}^{T}\norm{\nabla f(\x_t)}^2  \leq \varepsilon$, the query complexity of our algorithm is $\cO\left(\frac{dLG_u^2}{\varepsilon^2}\right)$. 
This complexity is the same as that of the valued-based ZO algorithm \citep{ghadimi2013stochastic}.

To the best of the author's knowledge, these query complexities are first obtained for rank-based ZO algorithms over stochastic functions. 
\end{remark}

\section{Conclusion}

This paper studied rank-based zeroth-order optimization for stochastic objectives,  motivated by machine learning, reinforcement learning, and learning systems that rely on ordinal feedback such as human preferences. We proposed a simple and computationally efficient rank-based ZO algorithm that operates purely on ranking information and avoids the quadratic-time graph constructions required by prior work.

Under standard assumptions including smoothness, strong convexity, and bounded second moments of stochastic gradients, we established explicit non-asymptotic query complexity bounds for both convex and nonconvex settings. These bounds match those of value-based zeroth-order methods, demonstrating that ordinal feedback alone is sufficient to achieve optimal query efficiency in stochastic optimization. Our analysis departs from existing drift-based and information-geometric techniques and provides tools that may extend to other rank-based algorithms.
\pb
\clearpage

\appendix

\section{Useful Lemmas}

\begin{lemma}[Lemma 15 of \citet{ye2025explicit}]\label{lem:low1}
	Let \(X_1,\dots,X_N\) be i.i.d.\ standard Gaussian,
	\(m=\lfloor N/4\rfloor\), and \(M:=X_{(N-m+1)}\), where $X_{(N-m+1)}$ means the $(N-m+1)$-th smallest in $X_i$'s.
	Given a fixed value $\tau$ and denoting $p=1-\Phi(\tau)$,   if \(p< q :=\tfrac{m}{N}\), then the Chernoff bound yields the exponential lower bound
	\begin{equation}\label{eq:chernoff-lower}
		\Pr(M>\tau)
		\;\ge\; 1-\exp\!\big(-N\,K(q \Vert p)\big),
	\end{equation}
	where the binary Kullback--Leibler divergence is
	\begin{equation}\label{eq:KL}
		K(q \Vert p) \;=\; q\ln\frac{q}{p}+(1-q)\ln\frac{1-q}{1-p}.
	\end{equation}
	In particular, for \(\tau=2\) we have \(p\approx 0.0224 < 1/4\) and hence (with \(q= 1/4\))
	\[
	\Pr(M>2) \;\ge\; 1-\exp\big(-N\,K(1/4\Vert p)\big).
	\]
\end{lemma}

\begin{lemma}[Lemma 17 of \citet{ye2025explicit}]\label{lem:low2}
	Let \(X_1,\dots,X_N\) be i.i.d.\ standard Gaussian,
	\(m=\lfloor N/4\rfloor\), and \(M:=X_{(m)}\), where $X_{(m)}$ means the $m$-th smallest in $X_i$'s.
	Given a fixed value $\tau$, and denoting  $p=\Phi(\tau)$, if \(p < q :=\tfrac{m}{N} \), then the Chernoff bound yields the exponential lower bound
	\begin{equation}\label{eq:chernoff-lower1}
		\Pr(M< \tau )
		\;\ge\; 1-\exp\!\big(-N\,K(q \Vert p)\big),
	\end{equation}
	where the binary Kullback--Leibler divergence $K(q \Vert p)$ is defined in Eq.~\eqref{eq:KL}.
	In particular, for \(\tau=-2\) we have \(p\approx 0.0224 <1/4 \) and hence 
	\[
	\Pr(M<-2) \;\ge\; 1-\exp\big(-N\,K(1/4\Vert p)\big).
	\]
\end{lemma}

\begin{lemma}[Gaussian Tail Bound \citep{vershynin2018high,laurent2000adaptive}]
	\label{lem:gaussian-tail}
	Let $X\sim\mathcal N(0,1)$. Then for any $\tau>0$,
	\[
	\Pr(|X|>\tau) \le 2 \exp\!\left( -\frac{\tau^2}{2} \right).
	\]
\end{lemma}

\begin{lemma}
	\label{lem:gauss_up}
	Let $X_1,\dots,X_N$ be i.i.d.\ $\mathcal N(0,1)$ variables, 
	$M_N=\max_{i\le N}|X_i|$.  Then 
	\[
	\Pr\!\left( M_N \le \sqrt{2\log\frac{2N}{\delta}} \right)
	\ge 1 - \delta, \mbox{ with } 0<\delta<1.
	\]
\end{lemma}

\begin{proof}
	By the union bound,
	\[
	\Pr(M_N>\tau)
	= \Pr\!\left( \bigcup_{i=1}^N \{|X_i|>\tau\} \right)
	\le \sum_{i=1}^N \Pr(|X_i|>\tau) \le 2N e^{-\tau^2/2},
	\]
	where the last inequality is because of Lemma~\ref{lem:gaussian-tail}.
	
	Setting the right-hand side of above equation less than $\delta$, we can obtain $\tau \ge \sqrt{2\log\frac{2N}{\delta}}$ which concludes the proof.
	
\end{proof}

\begin{lemma}\label{lem:u_up}[Lemma~27 of \citet{ye2025explicit}]
	Letting $\ub\sim \cN(0, \bm{I}_d)$ be a $d$-dimensional Gaussian vector, then with probability at least $1-\delta$, it  holds that
	\begin{equation}
		\label{eq:u_norm}
		\norm{\ub}^2 \le 2d + 3\log(1/\delta).
	\end{equation}
\end{lemma}

\begin{lemma}[Theorem 2.1.5 of \citet{nesterov2013introductory}]\label{lem:L_smth}
	Letting a function $f(\x)$ is $L$-smooth and  convex, then it holds that
	\begin{equation*}
		\norm{\nabla f(\x)}^2 \leq 2L \Big( f(\x) - f(\x^*)\Big).
	\end{equation*}
\end{lemma}

\begin{lemma}[Theorem 2.1.10 of \citet{nesterov2013introductory}]\label{lem:str_cvx}
	Letting a function $f(\x)$ is differentiable and $\mu$-strongly convex, then it holds that
	\begin{equation*}
		\norm{\nabla f(\x)}^2 \geq 2\mu \Big( f(\x) - f(\x^*)\Big).
	\end{equation*}
\end{lemma}

\begin{lemma}\label{lem:recs}
If it holds that $\Delta_{t+1} 
\leq 
\rho_t \Delta_t + \beta_t
$ for $0 \leq \rho_t $ and $1 \leq t$, then we have
\begin{equation*}
\Delta_{t+1}
\leq 
\prod_{i=k}^{t} \rho_i \Delta_{k} 
+ \sum_{i=k}^{t} \prod_{j=i+1}^{t} \rho_j \beta_i, \mbox{ with } 1 \leq k \leq t.
\end{equation*}
\end{lemma}
\begin{proof}
We prove the result by induction. 
For $k = t$, the result trivially holds.
We assume the result holds for $k\geq 2$, then we prove it will hold for $k+1$. 
\begin{align*}
\Delta_{t+1}
\leq& 
\prod_{i=k}^{t} \rho_i \Delta_{k} 
+ \sum_{i=k}^{t} \prod_{j=i+1}^{t} \rho_j \beta_i\\
\leq&
\prod_{i=k}^{t} \rho_i ( \rho_{k-1} \Delta_{k-1} + \beta_{k-1} )  
+ \sum_{i=k}^{t} \prod_{j=i+1}^{t} \rho_j \beta_i\\
=&
\prod_{i=k-1}^{t} \rho_i \Delta_{k-1}
+
\sum_{i=k-1}^{t} \prod_{j=i+1}^{t} \rho_j \beta_i,
\end{align*}
which concludes the proof.
\end{proof}

\begin{lemma}[proof of Lemma~2 of \citet{rakhlin2012making} ]\label{lem:dec2}
Letting $\Delta_{t+1} \leq \frac{b}{t(t-1)} \sqrt{ \sum_{i=2}^{t} (i-1)^2 \Delta_t} + \frac{c}{t}$ with $t \geq 2$, and $\Delta_t, b, c$ be non-negative, then it holds that
\begin{align*}
\Delta_t \leq \frac{a}{t}, \mbox{ with } a\geq \frac{9b^2}{4} + 3c,\mbox{ for all } t\geq 2.
\end{align*} 
\end{lemma}

\bibliography{ref}

\bibliographystyle{plainnat}
\end{document}